\documentclass[12pt]{article}
\usepackage{amsthm}
\usepackage{amsfonts}
\usepackage{amssymb}
\usepackage{amsmath}
\usepackage{amsthm}
\usepackage{nccmath}
\usepackage{mathtools}
\usepackage{mathrsfs}
\usepackage{setspace}
\usepackage{graphicx}
\usepackage{pgfplots}
\pgfplotsset{compat=newest}
\usepackage{subcaption}
\usepackage[export]{adjustbox}
\usepackage{hyperref}
\usepackage[numbers,sort&compress]{natbib}
\usepackage{xcolor}
\usepackage{enumerate}
\usepackage{authblk}
\usepackage{placeins}
\usepackage{caption}

\newtheorem{theorem}{Theorem}[section]
\newtheorem{lemma}[theorem]{Lemma}

\theoremstyle{definition}
\newtheorem{definition}[theorem]{Definition}
\newtheorem{example}[theorem]{Example}

\theoremstyle{remark}
\newtheorem{remark}[theorem]{Remark}

\numberwithin{equation}{section}

\usepackage{amsmath,amssymb,graphicx} 
\usepackage[margin=1in]{geometry} 
\usepackage{enumerate}
\usepackage{authblk}
\usepackage{placeins}
\usepackage{array}
\usepackage{academicons}
\usepackage{xcolor}
\usepackage{svg}
\usepackage[utf8]{inputenc}
\usepackage{booktabs}
\usepackage{tikz}
\usepackage{placeins}
\usetikzlibrary{arrows.meta}
\usepackage{float}
\providecommand{\keywords}[1]{\textbf{\textit{Keywords:}} #1}
\providecommand{\subjclass}[1]{\textbf{\textit{MSC2020:}} #1}

\begin{document}

\nocite{*} 

\title{Initial Value Problems for Some Functional Equations and the PENLP}

\author{Hailu Bikila Yadeta \\ email: \href{mailto:haybik@gmail.com}{haybik@gmail.com} }
  \affil{Salale University, College of Natural Sciences, Department of Mathematics\\ Fiche, Oromia, Ethiopia}
\date{\today}
\maketitle

\noindent
\begin{abstract}

\noindent In this paper, we study initial value problems for a class of functional equations. We introduce the concept of appropriate initial sets to enable the unique extension of an initial function into a solution defined over larger domains. Our analysis characterizes the structure and size of these initial sets, illustrating how functions specified on them can be uniquely extended to broader regions. Furthermore, we propose the \emph{Principle of Exclusion of Neighborhoods of Limit Points} (PENLP)— a novel conceptual framework that elucidates the behavior of solutions near limit points and their impact on the extension process. We demonstrate that the initial sets must avoid the limit points of the functional equations to ensure uniqueness and well-posedness. The functional equations examined in this study include:
$$ y(x+1)=  y(bx),\quad b \neq 0, $$
$$ y(x)=y(2x),\quad y(x)=y(-2x),\quad  y(3x)=y(x)+y(2x), $$
as well as the equations characterizing even functions and odd functions.
\end{abstract}

\noindent\keywords{ Functional equation, initial value problem for functional equation, limit points of FE, scaling type functional equations, shift type functional equations, mixed type functional equation, iterations }\\
\subjclass{Primary 39B06, 39B20,39B32}\\
\subjclass{Secondary 39A10, 39A05}

\section{Introduction}
According to Acz\'{e}l, a functional equation is an equation $A_1 = A_2$ between two terms $A_1$ and $A_2$, which involves $k$ independent variables $x_1, x_2, \ldots, x_k$ and $n \geq 1$ unknown functions $F_1, F_2, \ldots, F_n$, each depending on variables $j_1, j_2, \ldots, j_{n}$, along with a finite number of known functions. See \cite{JA}.

In this paper, we study initial value problems for certain functional equations in one independent variable $x$, involving shifts and scalings in $x$, with a single unknown function $y$. Specifically, we analyze the forms of initial sets such that, when initial functions are defined on these sets, the problem admits a unique solution—an extension of the initial function to a larger domain. The initial set is not unique but typically takes a specific form. In selecting this initial set, we avoid neighborhoods of certain points called limit points of the functional equation within the initial set. However, not all functional equations have limit points.

The principle of exclusion of neighborhoods of limit points (PENLP) is primarily essential for guaranteeing a unique solution to an initial value problem for a functional equation with an arbitrary initial function.

Consider the functional equation of all periodic functions of fundamental period equal to $1$,  which are given by
\begin{equation}\label{eq:1periodicequation}
   y(x+1)= y(x), \quad x \in\mathbb{ R },
\end{equation}
and whose general solution is given by $y(x) = \mu(x) $, where $\mu $ is arbitrary 1-periodic function. The constant function $y(x)= c $ can be accounted as a special solution of equation (\ref{eq:1periodicequation}), since a constant function is inherently periodic with any period. See, for example, \cite{AB}. If an initial function $f$ is set on some unit interval of the form $[x_0,x_0+1)$, we have a unique solution of the functional equation (\ref{eq:1periodicequation}), which is determined by the 1-periodic extension of $f$. As such the initial value problem
\begin{equation}\label{eq:IVPfor1periodic}
  y(x+1)=y(x), \, x \in  \mathbb{ R}, \quad  y(x)=f(x),  0 \leq x < 1,
\end{equation}
is given by
\begin{equation}\label{eq:solnforIVP1periodic}
  y(x)= f(\{x\}),
\end{equation}
where $ \{x\} $ is the fraction part of $x$.
If the initial value problem is given on some interval of length more than a unit-say, an interval of the form $[x_0,x_0+2) $, the initial value problem will be inconsistent unless and otherwise the value of the initial function $f$ is set in such a way that its value on $[x_0+1,x_0+2]$ matches with  the periodic extension of its value on $[x_0,x_0 +1)$. This requirement introduces redundancy in the data and is generally undesirable..

\begin{definition}
Let us define the set of all real-valued function whose domain is the set of real numbers $\mathbb{R}$ by
  \begin{equation}\label{eq:realvaluedfunctiononR}
    \mathcal{F}:= \{ f: \mathbb{R} \rightarrow        \mathbb{R}   \}
  \end{equation}
\end{definition}

\begin{definition}
  Define the \textbf{scaling operator} on a function  $y \in \mathcal{F }$ as
  \begin{equation}\label{eq:scaling}
    S^hy(x):=y(hx)
  \end{equation}
\end{definition}

We describe some properties of the scaling operator
\begin{itemize}
  \item $S^1=I $, where $I$ is the identity operator.
 \item  The $n$-fold composition of scaling operators yield:
  $$  S^{a_1}  \circ       S^{a_2} .... \circ S^{a_n}= S^{a_1a_2...a_n}  $$
 \item In particular,
  $$ S^{a}  \circ       S^{a}\circ .... \circ S^{a}= S^{a^n} . $$
\end{itemize}
    The \textbf{shift-scaling class of functional equations} includes:
 \begin{itemize}
   \item \textbf{Scaling types}: Only scaling operators are involved, for example,
   $$y(3x)=y(2x)+y(x). $$
   \item  \textbf{Shift type} : Only shift operators are involved in the equation, for example, $$y(x+2)=y(x+1)+y(x). $$
   \item  \textbf{ Mixed types} Both scaling shift operators are involved, for instance,
    $$y(x+1)=y(2x)+y(x) . $$
 \end{itemize}
In this paper, we investigate the initial value problem for a basic prototype class of shift-scaling functional equations. We propose a novel approach by identifying a suitable initial set on which the initial function, defined on this set, can be uniquely extended to a larger domain.

\section{The Functional Equation $y(x+1) = y(bx),\, b \neq 0 $}
In this section we consider the functional equation
\begin{equation}\label{eq:xplus1bx}
  y(x+1)= y(bx),\quad b \neq 0.
\end{equation}
 If we fix  $b=1$ in equation (\ref{eq:xplus1bx}), we get the usual homogeneous linear difference equation of the first order with constant coefficients $    y(x+1)=  y(x), $ that was discussed in the introduction section. In this section, we consider the cases $0 < b< 1 $, $b>1 $, and $b<0$. We formulate the possible set $S$ such that the initial value problem for (\ref{eq:xplus1bx}) has a unique well-defined solution when the initial data is set on $S$. Such a set may not unique. But we state the possible forms of the sets such that the initial value problem for the functional equation (\ref{eq:xplus1bx}) with  initial data defined on such  sets is well-posed. Let us denote the coordinates of the point of intersection  of the lines $y= x+1$ and $y=bx$ by
 $$( x^*,  bx^*) , \text{ where }   x^* =\frac{1}{b-1}. $$
 We have
 $$ \begin{cases}
       & x^* < 0,     \mbox{  if } 0 < b < 1, \\
       & x^* > 0,      \mbox{ if }  b > 1 .
    \end{cases}$$

\begin{figure}[h]
\centering
\begin{subfigure}[b]{0.32\textwidth}
\centering
\begin{tikzpicture}[scale=0.8]
\begin{axis}[
    axis lines = middle,
    xlabel = $x$,
    ylabel = $y$,
    xmin = -3, xmax = 3,
    ymin = -3, ymax = 3,
    grid = both,
    grid style = {line width = 0.1pt, draw = gray!20},
    title = {Case: $0 < b < 1$ },
    legend style = {at={(0.5,-0.2)}, anchor=north}
]

\addplot[domain=-3:3, thick, blue] {0.5*x};
\addplot[domain=-3:3, thick, red] {x + 1};

\addplot[only marks, mark=*, mark size=2pt, black] coordinates {(-2, -1)};

\legend{$y = bx$, $y = x + 1$}
\end{axis}
\end{tikzpicture}
\caption{Intersection in third quadrant}
\end{subfigure}
\hfill
\begin{subfigure}[b]{0.32\textwidth}
\centering
\begin{tikzpicture}[scale=0.8]
\begin{axis}[
    axis lines = middle,
    xlabel = $x$,
    ylabel = $y$,
    xmin = -3, xmax = 3,
    ymin = -3, ymax = 3,
    grid = both,
    grid style = {line width = 0.1pt, draw = gray!20},
    title = {Case: $b = 1$},
    legend style = {at={(0.5,-0.2)}, anchor=north}
]

\addplot[domain=-3:3, thick, blue] {x};
\addplot[domain=-3:3, thick, red] {x + 1};

\legend{$y = x$, $y = x + 1$}
\end{axis}
\end{tikzpicture}
\caption{Parallel lines (no intersection)}
\end{subfigure}
\hfill
\begin{subfigure}[b]{0.32\textwidth}
\centering
\begin{tikzpicture}[scale=0.8]
\begin{axis}[
    axis lines = middle,
    xlabel = $x$,
    ylabel = $y$,
    xmin = -3, xmax = 3,
    ymin = -3, ymax = 3,
    grid = both,
    grid style = {line width = 0.1pt, draw = gray!20},
    title = {Case: $b > 1$ },
    legend style = {at={(0.5,-0.2)}, anchor=north}
]

\addplot[domain=-3:3, thick, blue] {2*x};
\addplot[domain=-3:3, thick, red] {x + 1};

\addplot[only marks, mark=*, mark size=2pt, black] coordinates {(-1, 0)};

\legend{$y = bx$, $y = x + 1$}
\end{axis}
\end{tikzpicture}
\caption{Intersection in fourth quadrant}
\end{subfigure}

\caption{Intersection behavior of $y = bx$ and $y = x + 1$ for different values of $b$}
\label{fig:intersection_cases}
\end{figure}
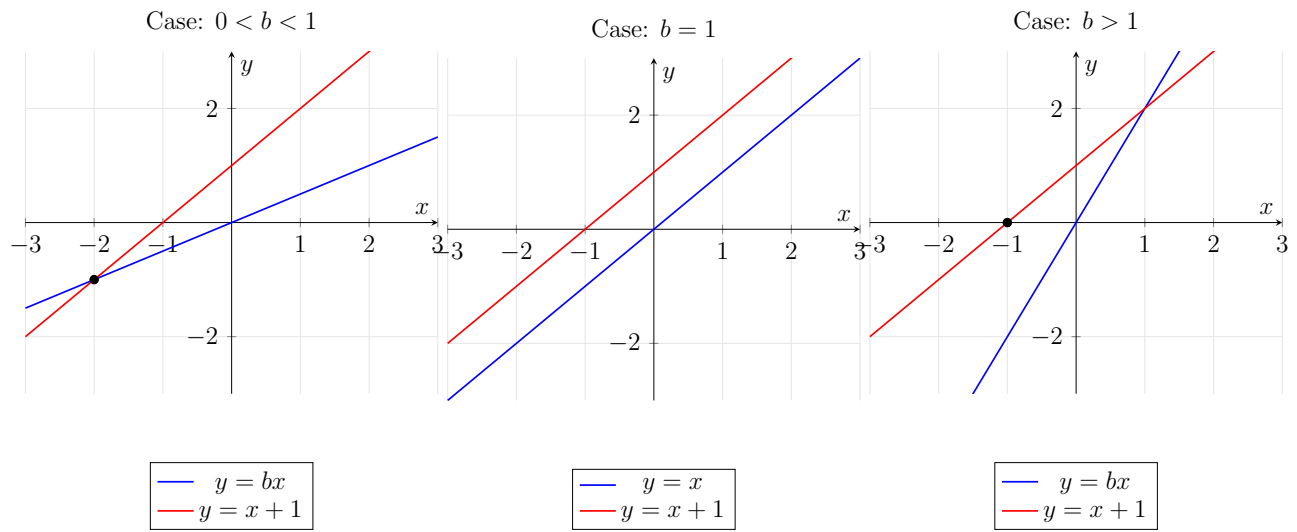

\subsection{The functional equation  $ y(x+1)=y(bx), 0 < b < 1 $}

\begin{theorem}
  For $0 < b< 1$, the forward recurrence relation with the initial point $x_0 > x^* $,
  \begin{equation}\label{eq:forwardrecurrence}
  x_{n+1} = \frac{1}{b} x_n + \frac{1}{b},
  \end{equation}
form a  strictly increasing  sequence that diverges to positive infinity.
 \end{theorem}

\begin{proof}
  The  solution of the forward recurrence relation (\ref{eq:forwardrecurrence}) with initial point $x_0$ is given by
  \begin{equation}\label{eq:soltoforwardrecurrence}
      x_n = b^{-n}(x_0-x^*)+x^*.
  \end{equation}
The  divergence and strict increasing conditions of the sequence are evident from equation (\ref{eq:soltoforwardrecurrence}), given that  $0 < b< 1$,and  $x_0 > x^* $.
\end{proof}

\begin{theorem}
  For $0 < b < 1$  the sequence determined by the initial point $x_0 > x^* $, and the backward recurrence relation
  \begin{equation}\label{eq:backwardrecurrence}
    x_{-(n+1)} = b x_{-n} -1
  \end{equation}
is strictly decreasing and it converge to $x^*$.
\end{theorem}

\begin{proof}
  The  solution of the backward recurrence relation (\ref{eq:backwardrecurrence}) with initial point $x_0$ is given by
  \begin{equation}\label{eq:soltobackwardrecurrence}
    x_{-n} = b^n (x_0-x^*)+x^*
  \end{equation}
 The convergence and strict decreasing conditions of the sequence are evident from the given conditions  $0 < b< 1$,  $x_0 > x^* $, and equation (\ref{eq:soltobackwardrecurrence}).
 \end{proof}

Let $x_n(x_0)$  is  the expression  given in (\ref{eq:soltoforwardrecurrence}), and  $x_{-n}(x_0)$  the expression given in (\ref{eq:soltobackwardrecurrence}). We have the invertibility relation:
\begin{equation}\label{eq:inversecomposition}
  x_{-k}(x_k(x))= x = x_k(x_{-k}(x)),\quad x \in \mathbb{R} .
\end{equation}

\begin{definition}
  Let $0< b< 1$. For $ k \in\mathbb{ Z} $ define the intervals:
\begin{align}\label{eq:Ikofxnotandb}
  I_k(x_0,b) & := [bx_k,  bx_{k + 1})= [x_{k-1}+1,  x_{k}+1 ) \nonumber \\
   & = \left[\frac{(x_0-x^*)}{b^{k-1}}+bx^*,\frac{(x_0-x^*)}{b^{k}}+ x^* +1\right).
\end{align}
\end{definition}

\begin{lemma}
 If $x_R \in I_{k}(x,b) $ then $ b(x_R-1) \in   I_{k-1}(x,0), \quad  k\in \mathbb{Z}. $
\end{lemma}

\begin{proof}
  Let $x_R \in  I_{k}(x_0,b) , k \in \mathbb{Z} $. Then by forward recurrence relation (\ref{eq:forwardrecurrence}) and backward recurrence relation (\ref{eq:backwardrecurrence}) we have
\begin{align*}
  bx_{k} & \leq x_R  < 1+ x_{k},   \\
  bx_{k}-1  & \leq   x_R-1 < x_{k},  \\
  x_{k-1} & \leq x_R-1  <  x_{k}, \\
  b x_{k-1} & \leq  b(x_R-1)  < b x_{k}.
\end{align*}
Therefore $ b(x_R-1) \in I_{k-1}(x_0,b)  $.
\end{proof}

\begin{lemma}
  If  $x_L \in  I_{-k}(x_0,b) , k \in \mathbb{Z}   $. Then $ \left(1+\frac{x_L}{b}\right)  \in   I_{1-k}(x_0,b)  $.
\end{lemma}

\begin{proof}
   Let $x_L \in  I_{-k}(x_0,b) , k \in \mathbb{N}   $. Then
\begin{align*}
  bx_{-k} & \leq x_L  <  bx_{1-k},  \\
  x_{-k}  & \leq \frac{x_L}{b}  <  x_{1-k},  \\
  1+ x_{-k} & \leq \left(1+\frac{x_L}{b}\right) < 1+ x_{1-k}, \\
  b x_{1-k} & \leq \left(1+\frac{x_L}{b}\right)  < b x_{2-k}.
\end{align*}
Therefore $ \left(1+\frac{x_L}{b}\right) \in I_{1-k}(x_0,b)  $.
\end{proof}

\begin{lemma}
 Let's denote length of the interval  $I_k(x_0,b)$  by $ l(I_k(x_0,b)) $. We have the following results:
  \begin{itemize}
    \item  The length of the interval  $I_k(x_0,b)$ is given by
    \begin{equation}\label{eq:lengthofIk}
   l(I_k(x_0,b))= \frac{(x_0-x^*)(1-b)}{b^k}.
\end{equation}
    \item We observe that $ I_k(x_0,b), k=0,1,2,... $ form a sequence of pairwise disjoint intervals  with
\begin{equation}\label{eq:rightunion}
     [x_0 +1 ,\infty)= [bx_1, \infty)= \bigcup_{k=1}^\infty I_k(x_0,b).
 \end{equation}
  \end{itemize}
 \end{lemma}

\begin{lemma}
 Let us denote length of the interval  $I_{-k}(x_0,b)$  by $ l(I_{-k}(x_0,b)) $. We have the following results:
\begin{itemize}
  \item The length of the interval  $I_{-k}(x_0,b)$  is given by:
\begin{equation}\label{eq:lengthofIminusk}
  l(I_{-k}(x_0,b))= b^k(1-b)(x_0-x^*).
\end{equation}
  \item We observe that $I_{-k}(x_0,b), k=1,2,... $ form a sequence of pairwise disjoint intervals  with
  \begin{equation}\label{eq:leftunion}
     (x^*, bx_0)= \bigcup_{k=1}^\infty I_{-k}(x_0,b).
 \end{equation}
\end{itemize}
\end{lemma}

\begin{theorem}
 Let $0 < b < 1 $. We have the following results:
\begin{itemize}
  \item  If $ x_0 > x^* $ then the initial value problem for the functional equation (\ref{eq:xplus1bx}), with  an initial function  $y(x)=y_0(x)$  defined on the interval  $ I_0(x_0,b):= [bx_0,  x_0 + 1) $,  has a unique solution on the interval $ (bx^*, \infty) $.

  \item If $x^* > x_0 $ then the initial value problem for the functional equation (\ref{eq:xplus1bx}), with  initial function  $y(x)=y_0(x)$  given on the interval  $ I_0(x_0,b):= [x_0 + 1,   bx_0) $  has a unique solution on the interval $ ( -\infty,   bx^*) $.

  \item Let $c < x^*, d > x^*$. Then for an initial functions,  $y(x)=f(x)$ defined on an interval $I_{0}:= [c+ 1, bc]  $, and  $y(x)=g(x)$ defined on an interval $ J_0 := [bd  , d+1)$, the initial value problem for(\ref{eq:xplus1bx}) has a unique solution on $ \mathbb{R} \setminus \{bx^*\}$.
\end{itemize}
\end{theorem}

\begin{proof}
Since the initial function,  $y(x)=y_0(x)$, is defined on the interval $ I_0(x_0,b):= [bx_0,  x_0 + 1) $, we need to calculate the values of $y$ on the intervals $[x_0+1, \infty) $, and  $(x^*, bx_0)$. Let $x_R \in [x_0+1,\infty)  $. By  (\ref{eq:rightunion}), $x_R \in I_k(x_0,b)$ for some $k\in \mathbb{N} $. From (\ref{eq:xplus1bx}) we have

\begin{equation}\label{eq:onestepleftevaluation}
   y(x_R) = y((x_R-1)+1)= y(b(x_R-1)).
\end{equation}
Hence we can continue the $k$-th backward iterate as follows:
$$ x_R \rightarrow b(x_R-1)\rightarrow b( (x_R-1)-1)\rightarrow...\rightarrow b^k(x_R-bx^*) + bx^* \in I_0(x_0,b).    $$
 Consequently,
 \begin{equation}\label{eq:kstepleftevaluation}
    y(x_R)=y_0\left(b^k(x_R-bx^*) + bx^* \right).
 \end{equation}
Now let us  solve for $y$ on the interval $(x^*, bx_0)$.

Let $x_L \in (x^*, bx_0) $. By  (\ref{eq:leftunion}), $x_L \in I_{-k}(x_0,b)$ for some $k\in \mathbb{N} $. From (\ref{eq:xplus1bx}) we have

\begin{equation}\label{eq:onesteprightevaluation}
   y(x_L) = y(b\left(\frac{x_L}{b}\right))= y( 1 +\frac{x_L}{b} ).
\end{equation}
Hence we can continue until the $k$-th right iterate as follows:
$$ x_L \rightarrow  \frac{x_L}{b} +1 \rightarrow  \frac{x_L+b+b^2}{b^2}   \rightarrow...\rightarrow \frac{x_L- bx^*}{b^k} + bx^* \in I_0(x_0,b).    $$
Consequently,
\begin{equation}\label{eq:ksteprightevaluation}
  y(x_L)= y_0\left( \frac{x_L-bx^*}{b^k}  +bx^* \right),\quad x_L \in I_{-k}(x_0,b).
  \end{equation}
Overall, from (\ref{eq:rightunion}) and (\ref{eq:leftunion}), (\ref{eq:kstepleftevaluation}), (\ref{eq:ksteprightevaluation}), and the initial data $y(x)= y_0(x)$ on $I_0(x_0,b) $, if $x \in (x^*, \infty)$ we have
\begin{equation}\label{eq:fromminusinftoinf}
  y(x)= \sum_{-\infty}^{\infty}y_0\left(b^n(x-bx^*)+bx^* \right)\chi_{I_n(x_0,b)}(x),
\end{equation}
where $  \chi_{I_n(x_0,b)}(x)$ is the characteristic function of the interval $I_n(x_0,b)  $. This completes the proof the  first proposition of the theorem where $x_0 >x^*$. The other two proportions follow similar argument. For illustration of the iterations and the intervals $I_k(x_0,b) $, see Fig 2.
\end{proof}

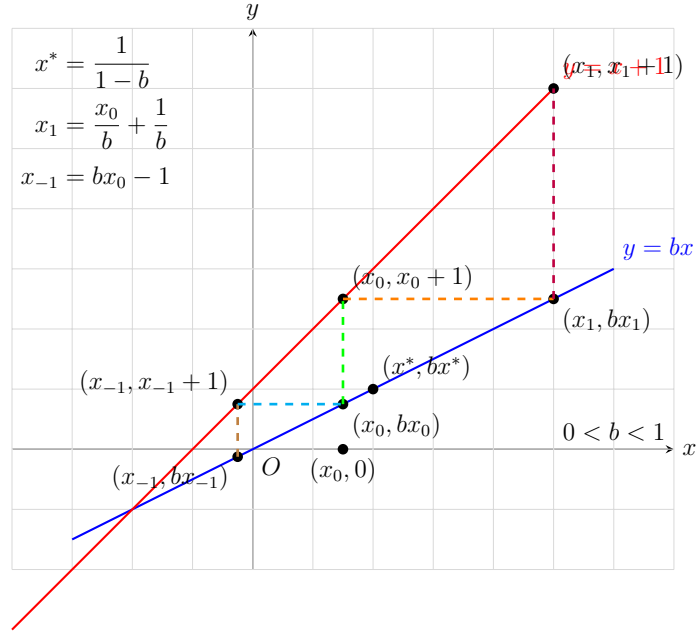
\begin{figure}[H]
\centering
\begin{tikzpicture}[scale=0.8, >=stealth, transform shape]

\def\b{0.5} 

\def\xzero{1.5}

\pgfmathsetmacro{\xstar}{1/(1-\b)}
\pgfmathsetmacro{\ystar}{\b*\xstar}

\pgfmathsetmacro{\xone}{\xzero/\b + 1/\b}
\pgfmathsetmacro{\xminusone}{\b*\xzero - 1}


\draw[->] (-4,0) -- (7,0) node[right] {$x$};
\draw[->] (0,-2) -- (0,7) node[above] {$y$};
\draw[very thin, gray!30] (-4,-2) grid (7,7);

\draw[thick, blue] (-3,{-3*\b}) -- (6,{6*\b}) node[above right] {$y = bx$};
\draw[thick, red] (-4,{-3}) -- (5,6) node[above right] {$y = x + 1$};

\fill (\xstar,\ystar) circle (2.5pt) node[above right] {$(x^*,bx^*)$};

\fill (\xzero,0) circle (2.5pt) node[below] {$(x_0,0)$};
\fill (\xzero,{\b*\xzero}) circle (2.5pt) node[below right] {$(x_0,bx_0)$};
\fill (\xzero,{\xzero+1}) circle (2.5pt) node[above right] {$(x_0,x_0+1)$};

\draw[dashed, green, line width=1pt] (\xzero,{\b*\xzero}) -- (\xzero,{\xzero+1});

\fill (\xone,{\b*\xone}) circle (2.5pt) node[below right] {$(x_1,bx_1)$};
\fill (\xone,{\xone+1}) circle (2.5pt) node[above right] {$(x_1,x_1+1)$};

\draw[dashed, orange, line width=1pt] (\xzero,{\xzero+1}) -- (\xone,{\b*\xone});

\draw[dashed, purple, line width=1pt] (\xone,{\b*\xone}) -- (\xone,{\xone+1});

\fill (\xminusone,{\b*\xminusone}) circle (2.5pt) node[below left] {$(x_{-1},bx_{-1})$};
\fill (\xminusone,{\xminusone+1}) circle (2.5pt) node[above left] {$(x_{-1},x_{-1}+1)$};

\draw[dashed, brown, line width=1pt] (\xminusone,{\b*\xminusone}) -- (\xminusone,{\xminusone+1});

\draw[dashed, cyan, line width=1pt] (\xminusone,{\xminusone+1}) -- (\xzero,{\b*\xzero});

\node[anchor=north west] at (0,0) {$O$};

\node[anchor=south east] at (7,0) {$0 < b < 1$};

\node[anchor=north west] at (-4,7)
{$\begin{aligned}
x^* &= \frac{1}{1-b} \\
x_1 &= \frac{x_0}{b} + \frac{1}{b} \\
x_{-1} &= bx_0 - 1
\end{aligned}$};
\end{tikzpicture}
\caption{Geometric construction showing the relationship between the lines $y = bx$ and $y = x + 1$ for $0 < b < 1$. The intersection point is $(x^*, bx^*)$ where $x^* = \frac{1}{1-b}$. Various vertical and horizontal segments connect points on these lines, illustrating the iterative relationship between successive points.}
\label{fig:lines_construction}
\end{figure}

\subsection{The Functional Equation  $ y(x+1)=y(bx),  b > 1 $}

In the next theorem, we consider the case of $b > 1$. The scenario is similar to the case $0 < b < 1$ , with the primary difference lying in the interval on which the initial data is defined. This is a consequence of the fact that $x_0 + 1 < b x_0$ if $x_0 > x^*$, and $x_0 + 1 > b x_0$ if $x_0 < x^*$ when $b > 1$. Conversely, if $0 < b < 1$, then $x_0 + 1 > b x_0$ when $x_0 > x^*$, and $x_0 + 1 < b x_0$ when $x_0 < x^*$.

\begin{theorem}
 Let $ b > 1 $. In this case $x^* > 0 $. We have the following results:
\begin{itemize}

\item If $x_0 > x^*  $, then the initial value problem for the functional equation (\ref{eq:xplus1bx}) with  initial function  $y(x)=y_0(x)$ is a given on some interval of the form $ I_0(x_0,b):= [x_0 + 1,   bx_0) $  has a unique solution on the interval $ (x^*,   \infty) $.

  \item  If $x^* < x_0 $, then the initial value problem for the functional equation (\ref{eq:xplus1bx}) with  initial function  $y(x)=y_0(x)$ is a given on some interval of the form $ I_0(x_0,b):= [bx_0,  x_0 + 1) $  has a unique solution on the interval $ (-\infty,   x^*) $.

  \item Let $ \alpha > x^* $, and $\beta < x^*$. Let $y$ is defined by the initial functions $y =f_1(x)$ and $y(x)= f_2(x)$, where
  $$f_1:[\beta + 1, b\beta )\to \mathbb{R},\quad  f_2: [\alpha + 1 , b\alpha )\to \mathbb{R }. $$
  Then the initial value problem for (\ref{eq:xplus1bx}) has a unique solution on $ \mathbb{R} \setminus \{bx^*\}$.
\end{itemize}
\end{theorem}

We summarize the solution of the initial value problem for the various cases of initial  intervals in the following table.

\begin{table}[H]
\centering
\begin{tabular}{|l|l|l|}
\hline
\multicolumn{3}{|c|}{$0 < b < 1$} \\ \hline
\multicolumn{1}{|c|}{$x_0 > x^*$} & $I_0 = [bx_0, x_0+1)$ & $I_{\text{max}} = (bx^*, \infty)$ \\ \hline
\multicolumn{1}{|c|}{$x_0 < x^*$} & $I_0 = (x_0 + 1, b x_0]$ & $I_{\text{max}} = (-\infty, bx^*) $ \\ \hline
\multicolumn{3}{|c|}{$b = 1$} \\ \hline
\multicolumn{1}{|c|}{Any $x_0$} & $I_0 = [x_0, x_0+1 )$ & $I_{\text{max}} = \mathbb{R}$ \\ \hline
\multicolumn{3}{|c|}{$b > 1$} \\ \hline
\multicolumn{1}{|c|}{$x_0 > x^*$} & $I_0 = [x_0+1, bx_0)$ & $I_{\text{max}} = [bx^*, \infty)$ \\ \hline
\multicolumn{1}{|c|}{$x_0 < x^*$} & $I_0 = (b x_0, x_0 + 1]$ & $I_{\text{max}} = (-\infty, bx^*)$ \\ \hline
\end{tabular}
\caption{The maximal  set $I_{\text{max}}  $  of existence of solution for the IVP for $y(x+1)=y(bx), b >0 $, with a initial data $ y(x)=y_0(x), x \in I_0 $ .}
\end{table}

In the next subsections we treat the case of $b<0$ for $-1 < b< 0$, $b< -1$ and $b=-1$ separately. Unlike the case of $b>0$ the iterates in the case  $b<0$ oscillate in sign and in magnitude. Therefore a proper consideration of such behave is required.

\subsection{The functional equation  $ y(x+1)=y(bx), -1 < b < 0 $ }
Let $ b < 0 $, let us define the counterclockwise and clockwise iterations as follows:

  \begin{definition}
     The \emph{clockwise iteration}  is the recurrence relation
    \begin{equation}\label{eq:clockwiseiteration}
      x_{n+1} = b x_n - b, \text{ where }   x_0 \neq bx^* \text{ is given}.
    \end{equation}
  \end{definition}

\begin{definition}
 The \emph{counterclockwise iteration} is  the recurrence relation defined as
    \begin{equation}\label{eq:counterclockwiseiteration}
      x_{n+1} = \frac{x_n}{b} + 1,  \text{ where }  x_0 \neq bx^* \text{ is given}.
    \end{equation}
\end{definition}
Consider the following arrangement of the functional equation (\ref{eq:xplus1bx})
  \begin{equation}\label{eq:spiralinward}
    y(x) = y((x-1)+1) =y(b(x-1)).
  \end{equation}
  From (\ref{eq:spiralinward}) we construct the following recurrence relation.
  \begin{equation}\label{eq:spiralinrecurrence}
      x_{n+1}= bx_{n}- b,\quad  x_0= x
  \end{equation}
  The solution for the recurrence relation (\ref{eq:spiralinrecurrence}) is given by
  \begin{equation}\label{eq:solntospiralinrecurrence}
    x_n = b^n (x-bx^*)+ bx^*.
  \end{equation}
To emphasize that the initial term of the recurrence relation (\ref{eq:spiralinrecurrence}) is $x$, we may write $x_n$  as $x_n(x)$. So, $x =x_0(x)$.

\begin{equation}\label{eq:bxstarislimitpoint}
  x_n(x) \rightarrow bx^*, \text{ as  } n  \rightarrow \infty,   \forall x \in \mathbb{R}.
\end{equation}
From (\ref{eq:spiralinward}) and (\ref{eq:spiralinrecurrence}), for each $x \in \mathbb{R} $  we have
  \begin{equation}\label{eq:levelsequenceinward}
    y(x_n(x)) =y(x),\quad  n \in \mathbb{N} .
  \end{equation}
 From (\ref{eq:solntospiralinrecurrence}) we have,
 \begin{equation}\label{eq:shrinkingRadius}
  |x_{n+1}-bx^*| =|b| |x_{n}-bx^*| .
 \end{equation}
 Equation shows that each iterate $x_{n+1} $ gets closer to the limit point $bx^*$ with a factor $0 < |b|< 1 $ than its proceeding iterate $x_n$.

\begin{lemma}\label{eq:L1}
  Let $ -1 < b < 0 $. If $x_n >  bx^* $, then  $x_{n+1} <  bx^* $.
\end{lemma}

\begin{proof}
Let $x_n >  bx^* $.
  \begin{align*}
     & b^n(x-bx^*)+bx^* > bx^* \\
     & b^n(x-bx^*) > 0  \\
     & b^{n+1}(x-bx^*) < 0  \\
     & b^{n+1}(x-bx^*)+ bx^* < bx^* \\
     & x_{n+1}(x) <  bx^*
  \end{align*}
 The lemma is proved. Similarly, we have the next lemma, which is analogous to this one.
\end{proof}

\begin{lemma}\label{eq:L2}
  Let $ -1 < b < 0 $. If $x_n <  bx^* $, then  $x_{n+1} > bx^* $.
\end{lemma}

\begin{figure}[H]
\centering
\begin{tikzpicture}[scale=1.4, >=stealth]

\def\b{-0.5} 
\def\xzero{1.5} 

\draw[->, thin] (-2.5,0) -- (3,0) node[right] {$x$};
\draw[->, thin] (0,-1.5) -- (0,3) node[above] {$y$};

\draw[thick, blue] (-2.5,-1.5) -- (2.5,3.5) node[above right] {$y = x + 1$};
\draw[thick, red] (-2.5,1.25) -- (3,{-1.5}) node[below right] {$y = bx, \, -1 < b <0 $};

\pgfmathsetmacro{\intersectx}{-2/3}
\pgfmathsetmacro{\intersecty}{\intersectx + 1}
\coordinate (I) at (\intersectx, \intersecty);

\fill[green] (I) circle (2pt);

\coordinate (P0) at (\xzero, 0);
\coordinate (P1) at (\xzero, {\b*\xzero});
\coordinate (P2) at ({\b*\xzero - 1}, {\b*\xzero});
\coordinate (P3) at ({\b*\xzero - 1}, {\b*(\b*\xzero - 1)});
\coordinate (P4) at ({\b*(\b*\xzero - 1) - 1}, {\b*(\b*\xzero - 1)});
\coordinate (P5) at ({\b*(\b*\xzero - 1) - 1}, {\b*(\b*(\b*\xzero - 1) - 1)});
\coordinate (P6) at ({\b*(\b*(\b*\xzero - 1) - 1) - 1}, {\b*(\b*(\b*\xzero - 1) - 1)});
\coordinate (P7) at ({\b*(\b*(\b*\xzero - 1) - 1) - 1}, {\b*(\b*(\b*(\b*\xzero - 1) - 1) - 1)});
\coordinate (P8) at ({\b*(\b*(\b*(\b*\xzero - 1) - 1) - 1) - 1}, {\b*(\b*(\b*(\b*\xzero - 1) - 1) - 1)});

\draw[thick, magenta, ->] (P0) -- (P1);
\draw[thick, magenta, ->] (P1) -- (P2);
\draw[thick, magenta, ->] (P2) -- (P3);
\draw[thick, magenta, ->] (P3) -- (P4);
\draw[thick, magenta, ->] (P4) -- (P5);
\draw[thick, magenta, ->] (P5) -- (P6);
\draw[thick, magenta, ->] (P6) -- (P7);
\draw[thick, magenta, ->] (P7) -- (P8);

\foreach \point in {P0,P1,P2,P3,P4,P5,P6,P7,P8} {
    \fill[black] (\point) circle (1.5pt);
}

\end{tikzpicture}

\caption{Cobwebbing diagram showing the iterative convergence process between the lines $ y = x + 1 $ and $ y = bx $, \, $ -1 < b < 0 $. The clockwise iteration is asymptotically convergent to the limit point $ x = bx^x $. Note that for $ -1 < b < 0 $, the clockwise iteration is asymptotically convergent, while the counterclockwise iteration diverges; the opposite case occurs when $ b < -1 $. For $b=-1$ the iteration is cyclic}
\label{fig:cobweb}

\end{figure}
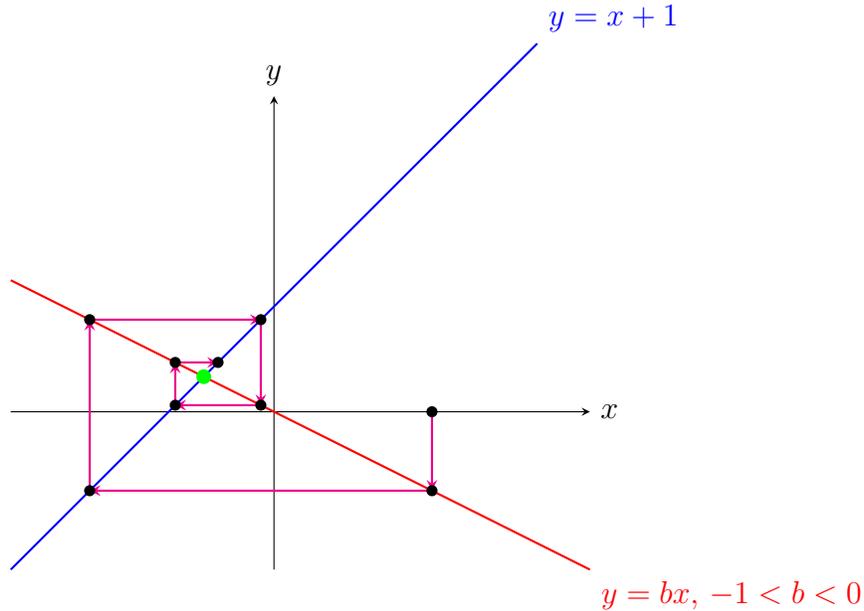

Based On  Lemmas (\ref{eq:L1}) and (\ref{eq:L2}), as well as  equation (\ref{eq:shrinkingRadius}), we observe that the iterates $x_n(x)$ oscillate around the limit point $bx^*$  getting closer and ultimately converging to it.
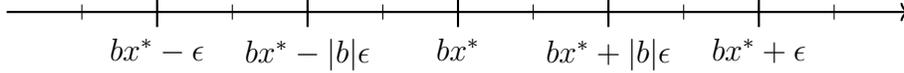
\begin{figure}[h]
\centering
\begin{tikzpicture}
\draw[->, thick] (-6,0) -- (6,0);

\draw[thick] (0,0.2) -- (0,-0.2) node[below] {$bx^*$};

\draw[thick] (2,0.2) -- (2,-0.2) node[below] {$bx^*+ |b|\epsilon$};
\draw[thick] (4,0.2) -- (4,-0.2) node[below] {$bx^*+\epsilon $};

\draw[thick] (-2,0.2) -- (-2,-0.2) node[below] {$bx^*-|b|\epsilon$};
\draw[thick] (-4,0.2) -- (-4,-0.2) node[below] {$bx^*-\epsilon$};

\foreach \x in {-5,-4,-3,-2,-1,0,1,2,3,4,5}
    \draw (\x,0.1) -- (\x,-0.1);
\end{tikzpicture}
\caption{Number line showing points around $bx^*$ with $\varepsilon$ spacing}
\label{fig:number_line}
\end{figure}

\begin{theorem}
 Let   $ -1 < b< 0  $  and $\epsilon > 0 $ be any positive number. Consider the initial value problem the functional equation (\ref{eq:xplus1bx}) with an initial function $y_0$ defined on the union of two intervals of the of the form
  \begin{equation}\label{eq:Intervalformforbminus1tozero}
   I_0(b, \epsilon) = (bx^*-\epsilon, bx^*-|b|\epsilon ] \cup [bx^*+ |b|\epsilon, bx^* + \epsilon ),
  \end{equation}
 excluding the neighbourhood of the limit point $bx^*$. Then there exists a unique solution to the initial value problem for the functional equation (\ref{eq:xplus1bx}) which is the unique extension of the initial function $y_0$ to $\mathbb{R} \setminus \{bx^*\}$.
\end{theorem}

\begin{proof}
 Since the initial function is defined on $(bx^*-\epsilon, bx^*-|b|\epsilon ] \cup [bx^*+ |b|\epsilon, bx^* + \epsilon ) $, it remains  to extend the initial function to the sets $ ( bx^* -|b|\epsilon , bx^*+|b|\epsilon  )\setminus    \{bx^*\}   $ and the complement $B(bx^*, \epsilon)^ \complement $. Suppose that $ x \in  B(bx^*, \epsilon)^ \complement    $.  Let $p$ be the smallest positive integer such that
 \begin{equation}\label{eq:thesmallestpostiveintgerinward}
   x_p(x) \in ( bx^*-|b|\epsilon,\,  bx^* + |b|\epsilon   )
 \end{equation}
 Equation (\ref{eq:thesmallestpostiveintgerinward}) implies that $-|b|\epsilon <  b^n(x-bx^*) < |b|\epsilon $.
 \begin{align*}
   \Rightarrow & \epsilon > b^{p-1}(x-bx^*) > - \epsilon  \\
   \Rightarrow & bx^* -  \epsilon < bx^* +  b^{p-1}(x-bx^*) < bx^* - \epsilon \\
   \Rightarrow  &  bx^* -  \epsilon < x_{p-1}  <  bx^* + \epsilon
  \end{align*}
  However $ x_{p-1} \notin        ( bx^*-|b|\epsilon,  bx^* + |b|\epsilon   ) $, by the assumption that $p$ is the smallest positive integer that $ x_p(x) \in ( bx^*-|b|\epsilon,  bx^* + |b|\epsilon   )$. Therefore $x_{p-1} (x)\in I_0(b,\epsilon) $.
  $$ y(x)= y_0(x_{p-1} (x)). $$
  Now consider the case $x \in      ( bx^*-|b|\epsilon,  bx^* + |b|\epsilon   ) \setminus \{bx^*\}$. Consider the rearrangement
  \begin{equation}\label{eq:spiraloutward}
    y(x)=y(b(x/b))= y((x/b)+1)
  \end{equation}
Equation (\ref{eq:spiraloutward}) induces the recurrence relation
  \begin{equation}\label{eq:spiraloutrecurrence}
    x_n(x)= x_{n-1}(x)/b + 1
  \end{equation}
  The solution for the recurrence relation (\ref{eq:spiraloutrecurrence}) is given by
  \begin{equation}\label{eq:solnforspiralout}
    x_n(x)=b^{-n}(x-bx^*)+bx^*
  \end{equation}
From (\ref{eq:spiraloutward}) and (\ref{eq:spiraloutrecurrence}), for each $x \in \mathbb{R} $  we have
  \begin{equation}\label{eq:levelsequenceoutward}
    y(x_n(x)) =y(x),\quad  n \in \mathbb{N} .
  \end{equation}
 From (\ref{eq:solnforspiralout}) we have,
 \begin{equation}\label{eq:expandingRadius}
  |x_{n+1}-bx^*| =|b|^{-1} |x_{n}-bx^*| .
 \end{equation}
  So each iterate $x_{n+1}(x)$ moves further from the limit point $bx^*$ with a factor $|b|^{-1}$ compared to its preceding element $x_{n}(x)$. Let $q$ be the smallest positive integer such that

  \begin{equation}\label{eq:thesmallestpostiveintgeroutward}
   x_n(x) \in B(bx^*,\epsilon)^\complement
 \end{equation}
  We claim that $x_{q-1}(x) \in I_0(b,\epsilon)$. Without loss of generality, let $x_q(x) > bx^* + \epsilon $.( Otherwise, it may happen that $x_q(x) < bx^* - \epsilon $). Then
  \begin{align*}
       & b^{-q}(x-bx^*)+bx^* > bx^*+\epsilon \\
     & \Rightarrow   b^{-q}(x-bx^*) > \epsilon \\
     & \Rightarrow     b^{1-q}(x-bx^*) <  -|b|\epsilon \\
     & \Rightarrow     bx^* + b^{1-q}(x-bx^*) < bx^* -|b|\epsilon            \\
     & \Rightarrow     x_{q-1}(x) < bx^* -|b|\epsilon .
  \end{align*}
  In addition, we must have  $ bx^* -\epsilon <  x_{q-1}(x)  $. Otherwise, $ x_{q-1}(x) \in B(bx^*,\epsilon)^\complement    $. This contradicts the assumption that $q$ is the smallest index such that $ x_{q}(x) \notin B(bx^*,\epsilon) $. Hence $x_{q-1} \in I_0(b,\epsilon) $. As a result,
  $$y(x)= y(x_{q-1}(x)) . $$
  Summarizing our results we have:
  \begin{equation}\label{eq:summarizedsolution}
   y(x)= \begin{cases}
       & y_0(x),  \mbox{ if } x \in I_0(b,\epsilon),  \\
       & y_0(x_{p-1}(x)),  \mbox{if }  x \in B(bx^*,\epsilon )^\complement, \\
       &  y_0(x_{q-1}(x)).    \mbox{ if } x \in B(bx^*,|b|\epsilon)\setminus \{bx^*\}.
    \end{cases}
  \end{equation}
 This complete the proof of the Theorem.
\end{proof}

\subsection{The functional equation $ y(x+1) = y(-x)$}

Consider the functional equation
\begin{equation}\label{eq:xplsoneminusx}
   y(x+1) = y(-x).
\end{equation}
Replacing $x$ by $x- \frac{1}{2} $, the equation (\ref{eq:xplsoneminusx}) can be rearranged as:
\begin{equation}\label{eq:plushalfminushalf}
  y \left(\frac{1}{2}-x \right)=y \left( \frac{1}{2} + x \right).
\end{equation}
Replacing $x$ by $-x$, the equation (\ref{eq:xplsoneminusx}) can be rearranged as:
\begin{equation}\label{eq:xoneminusx}
 y(x) = y(1-x),
\end{equation}
In this case, both the clockwise and the counterclockwise iterates are given by:
\begin{equation}\label{eq:ccwbequalsminus1}
  x_{n+1}=1-x_n
\end{equation}
The solution of  the recurrence equation (\ref{eq:ccwbequalsminus1}) with initial point $x$  is given by
\begin{equation}\label{eq:solntoccwbequalsminus1}
  x_n=(-1)^n\left(x-\frac{1}{2}\right) + \frac{1}{2}
\end{equation}
This is an oscillating sequence
$$ x, 1-x, x, 1-x,...   $$

\begin{theorem}
  The general solution of the functional equation $y(x+1)=y(-x)$ is given by  all function $y$ such that
  $$  y(x)= f(x-1/2), $$
 where $f$ is any even function.
\end{theorem}

\begin{proof}
Suppose $y(x)= f(x-1/2)$, where $f$ is an even function, that is, $f(t)=f(-t)$ for all $ t \in \mathbb{R} $. Then, for all $ x \in \mathbb{R} $,
 $$ y(-x)= f(-x-1/2)= f(x+1/2)= f((x+1)-1/2)= y(x+1) . $$
 This shows that such a $y$ satisfies the functional equation (\ref{eq:xplsoneminusx}).

Conversely, assume $y$ satisfies  the functional equation (\ref{eq:xplsoneminusx}) for all $x \in \mathbb{R} $. Define $ y(t):=f(t - 1/2) $. Then for all $x$, by (\ref{eq:plushalfminushalf}),
 $$ f(x)= y(x+1/2) = y(1/2-x) = f(-x).$$
 so $f$ is an even function. Therefore,
 $$  y(x) =f(x-1/2), $$
 where $f$ is any even function. This completes the proof.
\end{proof}

The general solution of the functional equation $y(x+1)=y(-x)$  consists of all functions whose graphs are symmetric about the vertical line $x=1/2$.  This symmetry arises because, for any point $x$, the point $1-x$ is the reflection across the line $x=1/2$. Let $\mathcal{S}$ be the collection of all sets  $ S \subset \mathbb{R} $ such that
\begin{itemize}
  \item $\frac{1}{2} \in S $ for all $ S \in \mathcal{S }$,
  \item  For all $ S \in  \mathcal{S}$, the condition $\frac{1}{2} -x \in S \Leftrightarrow  \frac{1}{2} + x \in S  $ holds.
\end{itemize}
 Consequently, the initial value problem for this functional equation can be specified using an initial set $ I$, such that $ 1/2 \in I $, and  either $ x \in I $ or $1-x \in I $. If  the initial function $y(x) =f(x), x \in I  $ is given, then it can be extended  from $I$ to $S $, where $S $ is the smallest set in  $\mathcal{S}$ containing $I$, by the consideration of symmetry.

\begin{example}
  Let   $x_0 > -\frac{1}{2}$. The  initial value problem for the functional equation (\ref{eq:xplsoneminusx}) with the initial function $y(x)=f(x)$ defined on $ S = \left[\frac{1}{2},x_0+1 \right )$,  can be uniquely extended to the interval $[-x_0, x_0 +1 ] $, by
  $$y(x)= \begin{cases}
            f(x) & \mbox{if } x \in \left[\frac{1}{2},x_0+1 \right ),  \\
            f(1-x)  & \mbox{  if } x \in (-x_0, \frac{1}{2} ) .
          \end{cases}$$
The solution can not be extended to the whole of $\mathbb{R}$ for any choice of a bounded initial set $S$.
\end{example}

\begin{theorem}
  Let $y$ be an even 1-periodic function of period $1$. The $y$ is the solution of the functional equation $y(x+1)=y(-x)$.
\end{theorem}

\begin{proof}
The proof follows from the following steps:
\begin{align*}
y(x+1) &= y(x) && \text{(by 1-periodicity)} \\
&= y(-x) && \text{(by evenness)}.
\end{align*}
\end{proof}

\begin{example}
  The $1$-periodic even function, $y(x)= \cos (2 \pi x) $, is a solution of the functional equation  (\ref{eq:xplsoneminusx}). Note that $y(x)= f(x-1/2)$, where $ f(x)= -\cos (2 \pi x)$.
\end{example}

\begin{theorem}
  Let $y$ be an odd $1$-antiperiodic function. Then $y$ is the solution of the functional equation (\ref{eq:xplsoneminusx}).
\end{theorem}

\begin{proof}
The proof follows from the following steps:
\begin{align*}
y(x+1) &= -y(x) && \text{(by 1-antiperiodicity)} \\
&= y(-x) && \text{(by oddness)}.
\end{align*}
\end{proof}

\begin{example}
The $1$-antiperiodic odd function, $y(x)= \sin (\pi x) $,  is a solution of the functional equation (\ref{eq:xplsoneminusx}). Note that $y(x)= f(x-1/2)$, where $ f(x)= \cos ( \pi x)$.
\end{example}

\subsection{ The functional equation  $ y(x+1)= y(bx), b < -1 $ }

\begin{theorem}
  Let $b < -1$. The initial value problem for the functional equation (\ref{eq:xplus1bx})  with initial value $y = y_0(x)$ given on a set
  $$  I_0(\epsilon, b):= ( bx^* -|b|\epsilon ,  bx^* -\epsilon ]  \cup [  bx^* + \epsilon ,  bx^* + |b|\epsilon   )$$
   has a unique solution on $\mathbb{R} \setminus \{x^*\}$.
\end{theorem}

\begin{proof}
We proceed similar to the case of $ -1< b < 0 $. However In this case the counterclockwise iteration (\ref{eq:counterclockwiseiteration}) oscillates in sign about  the limit point $bx^*$ converges to it. For $ x \in\mathbb{ R } \setminus   B(bx^*, |b|\epsilon) $, we have a unique $r \in \mathbb{N} $ such that
$$ b^{-r} (x-bx^*)+bx^* \in I_0(x,\epsilon). $$
Consequently,
$$  y(x) = y(b^{-r} (x-bx^*)+bx^*)= y_0(b^{-r} (x-bx^*)+bx^*).   $$
The clockwise iteration  (\ref{eq:clockwiseiteration}) diverges out while oscillating in sign . For $  B(bx^*, |b|\epsilon) \setminus \{ bx^* \}$, there exists a unique a unique $s \in \mathbb{N} $ such that
$$ b^{s} (x-bx^*)+bx^* \in I_0(x,\epsilon). $$
Thus,
$$  y(x) = y(b^{s} (x-bx^*)+bx^*)= y_0(b^{s} (x-bx^*)+bx^*) .  $$
\end{proof}

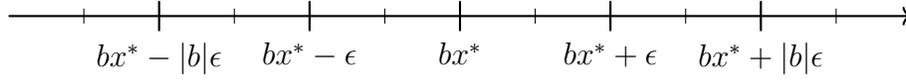
\begin{figure}[h]
\centering
\begin{tikzpicture}
\draw[->, thick] (-6,0) -- (6,0);

\draw[thick] (0,0.2) -- (0,-0.2) node[below] {$bx^*$};

\draw[thick] (2,0.2) -- (2,-0.2) node[below] {$bx^*+\epsilon$};
\draw[thick] (4,0.2) -- (4,-0.2) node[below] {$bx^*+|b|\epsilon$};

\draw[thick] (-2,0.2) -- (-2,-0.2) node[below] {$bx^*-\epsilon$};
\draw[thick] (-4,0.2) -- (-4,-0.2) node[below] {$bx^*-|b|\epsilon$};

\foreach \x in {-5,-4,-3,-2,-1,0,1,2,3,4,5}
    \draw (\x,0.1) -- (\x,-0.1);
\end{tikzpicture}
\caption{Number line showing points around $bx^*$ with $\varepsilon$ spacing}
\label{fig:number_line}
\end{figure}

\begin{table}[H]
\centering
\begin{tabular}{|l|l|l|}
\hline
\multicolumn{3}{|c|}{$0 > b > -1$} \\ \hline
\multicolumn{1}{|c|}{$x_0 > x^*$} & $I_0 = [bx_0, bx^*- \epsilon)\cup (bx^*+\epsilon, x_0+1)    $ & $I_{\text{max}} =  \mathbb{R} \setminus \{ bx^*\}$ \\ \hline
\multicolumn{1}{|c|}{$x_0 < x^*$} & $I_0 = (x_0 + 1, b x^*-\epsilon ) \cup [ bx^* + \epsilon, bx_0 ) $ & $I_{\text{max}} = \mathbb{R} \setminus \{ bx^*\}] $ \\ \hline
\multicolumn{3}{|c|}{$b = -1$} \\ \hline
\multicolumn{1}{|c|}{Any $ c \neq 0 $} & $I_0 = \{\frac{1}{2} +c \}$ & $I_{\text{max}} = \{\frac{1}{2} +c, \frac{1}{2} - c  \}$ \\ \hline
\multicolumn{3}{|c|}{$b < -1$} \\ \hline
\multicolumn{1}{|c|}{$x_0 > x^*$} & $ I_0 = [bx_0, bx^*- \epsilon)\cup (bx^*+\epsilon, x_0+1 ) $ & $I_{\text{max}} = \mathbb{R} \setminus \{bx^* \}$ \\ \hline
\multicolumn{1}{|c|}{$x_0 < x^*$} & $I_0 = [ x_0+1, bx^*-\epsilon) \cup (bx^*+\epsilon, bx_0 )$ & $I_{\text{max}} = \mathbb{R} \setminus \{bx^* \}$ \\ \hline
\end{tabular}
\caption{The maximal  set $I_{\text{max}}  $  of existence of solution for the IVP for $y(x+1)=y(bx), b<0 $, with a initial data $ y(x)=y_0(x), x\in I_0 $ .}
\end{table}



\section{The functional equations $y(x)=y(2x)$ and $y(x)=y(-2x)$}

\subsection{ The functional equation $y(x)= y(2x)$  }
\begin{theorem}
 Consider the  functional equation
 \begin{equation}\label{eq:2x1x}
   y(x)=y(2x)
 \end{equation}
The equation (\ref{eq:2x1x}) is consistent.
\end{theorem}
\begin{proof}
  Any constant function $y(x)=c $  is a solution of (\ref{eq:2x1x}) .
\end{proof}
\begin{theorem}
  The only  solutions $f$ of the functional equation(\ref{eq:2x1x}) that are continuous at $x=0$  are the  constant functions $f(x)=c $.
\end{theorem}

\begin{proof}
  Let $y(x)= f(x) $ is a solution of the functional equation(\ref{eq:2x1x}) that is continuous at $x=0$. Then
  $$ f(x)= f\left(\frac{x}{2^n}\right),\quad   \forall   n \in \mathbb{N},\quad \forall x \in \mathbb{R}  $$
  Applying the limit as $ n \rightarrow \infty $ and by continuity of $f$ at $x=0$,
  $$  f(x)= \lim_{n\rightarrow \infty}  f\left(\frac{x}{2^n}\right)=   f\left(  \lim_{n\rightarrow \infty}.     \frac{x}{2^n}\right) = f(0). $$
 \end{proof}

  \begin{remark}
    The functional equation (\ref{eq:2x1x}) has a closed-form solution
    $$ y(x)= \mu (\ln (x)), \quad \mu \in \mathbb{P}_{\ln 2}. $$
    This is a particular case of the more general functional equation:
    $$ y(ax)-by(x)=0, a>0,b>0. $$
    whose solutions are given by
    $$ y(x)= \mu (\ln (x)) x^{\lambda}, $$
    where $\lambda = \frac{\ln b }{\ln a}, \mu \in \mathbb{P}_{\ln a} $. See, for instance, (\cite{EW}).
  \end{remark}

  \begin{theorem}
    Let $\epsilon $ be any positive number. Then the initial value problem for the functional equation (\ref{eq:2x1x}) with initial function
    \begin{equation}\label{eq:IVPfor2x1x}
       y(x)=f(x), x \in  [\epsilon, 2\epsilon ),
    \end{equation}
    has a unique solution on the interval $(0, \infty)$.
  \end{theorem}

  \begin{proof}
    From the initial data, $   y(x)=f(x), x \in  [\epsilon, 2\epsilon) $. Let $x \in [2\epsilon,  \infty) $. Since

    \begin{equation}\label{eq:2epsilontoinfinity}
      [2\epsilon,  \infty) =    \bigcup_{n=1}^\infty [2^n \epsilon, 2^{n+1}\epsilon ),
    \end{equation}
    which is a union countable disjoint intervals. Therefore $\frac{x}{2^n} \in     [\epsilon, 2\epsilon) $ for a  unique $n\in \mathbb{N} $. Consequently,
    $$ y(x)= y\left(\frac{x}{2^n}\right)=f\left(\frac{x}{2^n}\right). $$
    On the other hand
    \begin{equation}\label{eq:0toepsilon}
       (0,\epsilon) =  \bigcup_{n=1}^\infty \left[\frac{\epsilon}{2^n}, \frac{\epsilon}{2^{n-1}},
     \right)
    \end{equation}
 which is a union of disjoint intervals. Therefore, for any $x \in (0,\epsilon) $ there exists a unique $n\in \mathbb{N} $ such that $x2^n \in    [\epsilon, 2\epsilon) $. Consequently,
     $$y(x)= y(2^n x) =  f(2^nx) $$
    Therefore we have unique solution of the initial value problem for the functional equation (\ref{eq:IVPfor2x1x}) on $(0,\infty )$.
  \end{proof}
\begin{example}
Consider the IVP for (\ref{eq:2x1x}), with initial data $y(x)=x, 1 \leq x < 2 $. The solution is then given by:
  $$
  y(x) =\begin{cases}
                        & 2^nx, \frac{1}{2^n} \leq x < \frac{1}{2^{n-1}}\\
                        &\vdotswithin{ = } \\[-1ex]\\
                        & 2x, \quad  \mbox{ if }\quad  \frac{1}{2} \leq x <1,  \\
                        & x , \quad  \mbox{if }\quad  1 \leq x < 2,  \\
                        & \frac{1}{2}x,\quad  \mbox{ if } \quad  2 \leq  x < 4,  \\
                        &\vdotswithin{ = } \\[-1ex]
                        & \frac{1}{2^n}x,\quad  \mbox{ if } \quad  2^n  \leq  x < 2^{n+1}, n \in \mathbb{N}   .
  \end{cases}
  $$
\end{example}

 \begin{theorem}
    Let $\delta $ be any positive number. Then the initial value problem for the functional equation (\ref{eq:2x1x}) with initial function
    \begin{equation}\label{eq:IVPfor2x1x}
       y(x)=f(x),\quad  x \in I_0:= (-2\delta, -\delta],
    \end{equation}
    has a unique solution on the interval $ (-\infty,0) $.
  \end{theorem}

 \begin{theorem}
    Let $\epsilon, \delta $ be any positive numbers. Then the initial value problem for the functional equation (\ref{eq:2x1x}) with initial function
    \begin{equation}\label{eq:IVPfor2x1x}
       y(x)=f(x),\quad  x \in I_0:= (-2\delta, -\delta] \cup  [\epsilon, 2\epsilon) ,
    \end{equation}
    has a unique solution on the set  $ \mathbb{R} \setminus \{0\}$.
  \end{theorem}

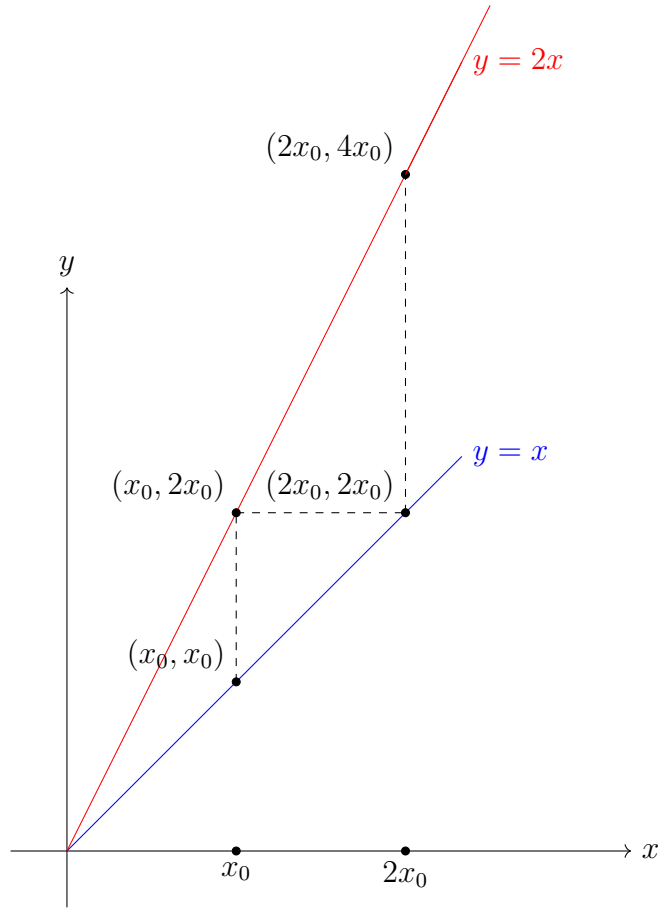
\begin{figure}[H]
\centering
\begin{tikzpicture}[scale=1.5]
    \draw[->] (-0.5,0) -- (5,0) node[right]{$x$};
    \draw[->] (0,-0.5) -- (0,5) node[above]{$y$};

    \draw[domain=0:3.5, smooth, variable=\x, blue] plot ({\x}, {\x}) node[right]{$y=x$};
    \draw[domain=0:3.5, smooth, variable=\x, red] plot ({\x}, {2*\x}) node[right]{$y=2x$};

    \coordinate (x0) at (1.5,0); 
    \filldraw (x0) circle (1pt) node[below]{$x_0$};

    \coordinate (x0x0) at (1.5,1.5);
    \coordinate (x02x0) at (1.5,3);
    \filldraw (x0x0) circle (1pt) node[above left]{$(x_0,x_0)$};
    \filldraw (x02x0) circle (1pt) node[above left]{$(x_0,2x_0)$};
    \draw[dashed] (x0x0) -- (x02x0);

    \coordinate (2x0) at (3,0);
    \coordinate (2x02x0) at (3,3);
    \coordinate (2x04x0) at (3,6);
    \filldraw (2x0) circle (1pt) node[below]{$2x_0$};
    \filldraw (2x02x0) circle (1pt) node[above left]{$(2x_0,2x_0)$};
    \filldraw (2x04x0) circle (1pt) node[above left]{$(2x_0,4x_0)$};

    \draw[dashed] (x02x0) -- (2x02x0) node[midway,above] {};
    \draw[dashed] (2x02x0) -- (2x04x0);

    \draw[domain=3:3.75, smooth, variable=\x, red] plot ({\x}, {2*\x});
\end{tikzpicture}
\caption{Construction showing relationships between points on $y=x$ and $y=2x$.
All coordinates are expressed in terms of an arbitrary $x_0 > 0$.
The dashed lines connect corresponding points in the construction.}
\end{figure}

%

\begin{theorem}
    Consider the functional equation (\ref{eq:2x1x}), with initial value $y_0$  defined on  an interval $[a,b),\,0< a < b $. Then the following  possible cases arise:
  \begin{itemize}
  \item If $ b=2a $, then the initial value problem  has unique solution on $(0,\infty) $.
    \item If $ 2a < b $, the initial value problem is \textbf{overdetermined}. That is, while the initial data given on interval $[a,2a)$ is sufficient to calculate $y$ on $(0,\infty)$,  the extra data given on $[2a,b)$ create an inconsistency, unless for the exceptional case where the calculated value of $y$ on $[2a,b)$ coincides with the given data on $[2a,b)$.
    \item If $ b < 2a  $ then the initial value problem is \textbf{underdetermined}. That is  the initial data $y_0$ defined on $[a,b)$ is not enough to guarantee the unique solution of the functional equation on $(0,\infty) $. For different assignments of additional data on $[b,2a)$ yields a different solution on $(0,\infty) $. The maximum interval of existence is:
         $$I_{\text{max}} = \bigcup_{n=-\infty}^{\infty} [2^na,2^nb) \subsetneqq (0,\infty). $$
        \end{itemize}
\end{theorem}

\subsection{The functional equation $y(x)=y(-2x)$ }

In this subsection, we consider the functional equation:
\begin{equation}\label{eq:xminus2x}
  y(x)=y(-2x).
\end{equation}
For the functional equation (\ref{eq:xminus2x}), we draw on the ideas from the functional equation(\ref{eq:2x1x}) that we have previously discussed. The main difference is that, for (\ref{eq:xminus2x}) the initial set $I_0 $ should be considered differently .


\begin{theorem}
  Let $\epsilon > 0 $.  The functional equation (\ref{eq:xminus2x}) with initial function $y_0$ defined on the set of the form $ I_0:= (-2 \epsilon, -\epsilon]  \cup [\epsilon, 2\epsilon)  $ has a unique solution on $\mathbb{ R} \setminus \{0\}$ which is the unique extension of $y_0 $.
\end{theorem}
\begin{proof}
  We can write the functional equation (\ref{eq:xminus2x}) as
  \begin{equation}\label{eq:outward}
    y(x)=y((-2)^n x), n  \in \mathbb{N}.
  \end{equation}
  We use the formulation (\ref{eq:outward}) to extend the initial function $y_0$ to a newer set $(-\epsilon, \epsilon)\setminus \{0\}$. Let $x \neq 0, x \in (-\epsilon,  \epsilon) $. Then $|x| \in (0, \epsilon)$. By the condition given in (\ref{eq:0toepsilon}),  there exists a unique $n \in \mathbb{N} $ such that $ 2^n |x| \in  [\epsilon,  2\epsilon) $. Consequently, $ (-2)^n x \in (-2 \epsilon, -\epsilon]  \cup [\epsilon, 2\epsilon)  $, and by (\ref{eq:outward})
  $$  y(x) = y((-2)^n x)= y_0((-2)^n x) $$.
 We can also write the functional equation (\ref{eq:xminus2x}) as
  \begin{equation}\label{eq:inward}
    y(x)= y(\left(\frac{-1}{2}\right)^n x ), n  \in \mathbb{N}
  \end{equation}
  We use the formulation (\ref{eq:inward}) to extend the initial function $y_0$ to a newer set$(-\infty, -2\epsilon] \cup [2\epsilon, \infty )$.

 Let  $x \in  (-\infty, -2\epsilon]\cup [2\epsilon, \infty) $. Then $|x| \in [2\epsilon, \infty)  $. By (\ref{eq:2epsilontoinfinity})   there exists a unique $ m \in \mathbb{N }$ such that $\left(\frac{1}{2}\right)^m |x | \in  (\epsilon,  2\epsilon) $. Consequently, $ (-\frac{1}{2})^m x \in (-2 \epsilon, -\epsilon]  \cup [\epsilon, 2\epsilon)  $. By (\ref{eq:inward}), we get
  $$y(x) = y(\left(-\frac{1}{2}\right)^m x)= y_0(\left(-\frac{1}{2}\right)^m x) .$$

  Summarizing these results we write as:
  $$ y(x)= \begin{cases}
              &  y_0(x), \mbox{ if } x \in (-2\epsilon,  -\epsilon]\cup [\epsilon, 2\epsilon), \\
              &   y_0((-2)^n x), \mbox{  if }  x \in (-\epsilon, \epsilon)\setminus \{0\}, \\
              &    y_0(\left(-\frac{1}{2}\right)^m x),     \mbox{ if }    x\in (-\infty, -2\epsilon)\cup (2\epsilon, \infty).
           \end{cases}  $$
  This  completes the proof of the theorem.
\end{proof}

\begin{remark}
  Note that if the functional equation (\ref{eq:xminus2x}) is given with initial function defined on an interval of the form $[\epsilon, 2\epsilon)$  is helpful to  extend the initial function to collection of intervals of the form
  $$    \{  [2^{2n}\epsilon, 2^{2n+1}\epsilon     )\}, ( -2^{2n}\epsilon , 2^{2n-1}\epsilon,    ] , n \in \mathbb{Z }\}  $$
  where as for  initial function defined on an interval of the form $(-2\epsilon, -\epsilon]$ we can extend the initial function to the collection of intervals of the form:

  $$    \{  (-2^{2n+1}\epsilon, -2^{2n}\epsilon     ]\}, ( 2^{2n-1}\epsilon , 2^{2n}\epsilon,    ] , n \in \mathbb{Z }\}  $$
  The union of these two collections yield $\mathbb{R} \setminus \{0\}$. If we
\end{remark}

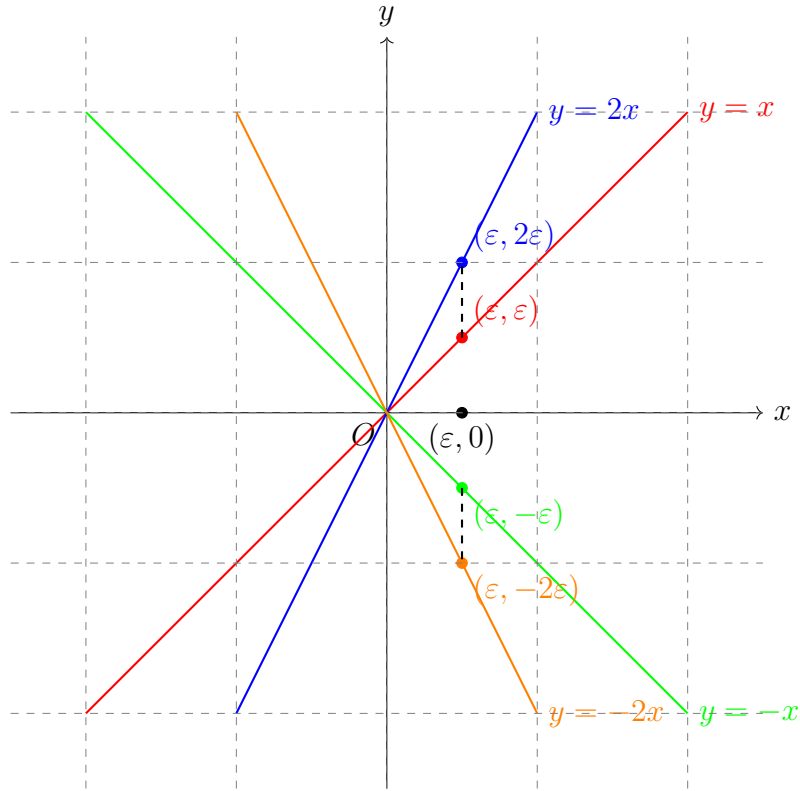
\begin{figure}[H]
\centering
\begin{tikzpicture}[scale=2]

\draw[->] (-2.5,0) -- (2.5,0) node[right] {$x$};
\draw[->] (0,-2.5) -- (0,2.5) node[above] {$y$};

\node at (0,0) [below left] {$O$};

\draw[red, thick] (-2,-2) -- (2,2) node[right] {$y=x$};
\draw[blue, thick] (-1,-2) -- (1,2) node[right] {$y=2x$};
\draw[green, thick] (-2,2) -- (2,-2) node[right] {$y=-x$};
\draw[orange, thick] (-1,2) -- (1,-2) node[right] {$y=-2x$};

\coordinate (eps) at (0.5,0);
\filldraw (eps) circle (1pt) node[below] {$(\varepsilon,0)$};

\coordinate (p1) at (0.5,0.5); 
\coordinate (p2) at (0.5,1);   
\coordinate (p3) at (0.5,-0.5); 
\coordinate (p4) at (0.5,-1);   

\filldraw[red] (p1) circle (1pt) node[above right] {$(\varepsilon,\varepsilon)$};
\filldraw[blue] (p2) circle (1pt) node[above right] {$(\varepsilon,2\varepsilon)$};
\filldraw[green] (p3) circle (1pt) node[below right] {$(\varepsilon,-\varepsilon)$};
\filldraw[orange] (p4) circle (1pt) node[below right] {$(\varepsilon,-2\varepsilon)$};

\draw[thick, dashed] (p1) -- (p2);
\draw[thick, dashed] (p3) -- (p4);

\draw[gray, very thin, dashed] (-2.5,-2.5) grid (2.5,2.5);

\end{tikzpicture}
\caption{Nature of initial set $I_0$. The figure shows the vertical segments connecting points at $x = \varepsilon$ on the lines $y = x$ to $y = 2x$ (above the x-axis) and $y = -x$ to $y = -2x$ (below the x-axis), illustrating the structure of the initial set.}
\label{fig:initial_set}
\end{figure}

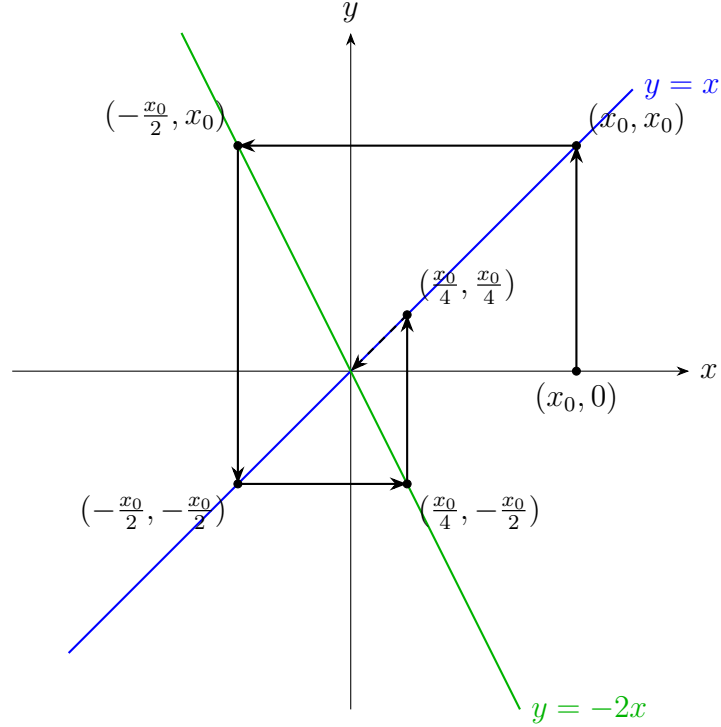
\begin{figure}[htbp]
\centering
\begin{tikzpicture}[>=Stealth, scale=1.5]
    \draw[->] (-3,0) -- (3,0) node[right] {$x$};
    \draw[->] (0,-3) -- (0,3) node[above] {$y$};

    \coordinate (X0) at (2,0);  
    \coordinate (X0X0) at (2,2); 
    \coordinate (NegX02X0) at (-1,2); 
    \coordinate (NegX02NegX02) at (-1,-1); 
    \coordinate (X04NegX02) at (0.5,-1); 
    \coordinate (X04X04) at (0.5,0.5); 

    \draw[blue, thick] (-2.5,-2.5) -- (2.5,2.5) node[right] {$y=x$};
    \draw[green!70!black, thick] (-1.5,3) -- (1.5,-3) node[right] {$y=-2x$};

    \filldraw (X0) circle (1pt) node[below] {$(x_0,0)$};
    \filldraw (X0X0) circle (1pt) node[above right] {$(x_0,x_0)$};
    \draw[->, thick] (X0) -- (X0X0);
    \filldraw (NegX02X0) circle (1pt) node[above left] {$(-\frac{x_0}{2},x_0)$};
    \draw[->, thick] (X0X0) -- (NegX02X0);
    \filldraw (NegX02NegX02) circle (1pt) node[below left] {$(-\frac{x_0}{2},-\frac{x_0}{2})$};
    \draw[->, thick] (NegX02X0) -- (NegX02NegX02);
    \filldraw (X04NegX02) circle (1pt) node[below right] {$(\frac{x_0}{4},-\frac{x_0}{2})$};
    \draw[->, thick] (NegX02NegX02) -- (X04NegX02);
    \filldraw (X04X04) circle (1pt) node[above right] {$(\frac{x_0}{4},\frac{x_0}{4})$};
    \draw[->, thick] (X04NegX02) -- (X04X04);

    \draw[dashed, ->, thick] (X04X04) -- (0,0);
\end{tikzpicture}
\caption{Counterclockwise (follow the arrow) sense $y(x)=y(-2x)$ showing spiraling inward to the limit point. The clockwise sense(reverse the arrow) is spiraling outward}
\end{figure}

\section{The functional equations of all even functions and all odd functions}

\subsection{The functional equation of all even function }

\begin{theorem}
Consider a  set $I=(-a,a), a> 0 $,  or  $ I = \mathbb{R}$ when $a= \infty $ which are symmetric about the point $0$.
Let $\mathcal{C} $ is a family of all subsets $S$  of $I$ satisfying the following conditions:

  \begin{itemize}
    \item For each  $ S \in \mathcal{C},  0 \in S $
    \item For each  $S \in \mathcal{C} $, for each   $x \in I $  either $x \in S $ or $-x \in S $.
  \end{itemize}
  Then  for any arbitrary function $y^0$  defined on some given set  $S \in \mathcal{C} $, the initial value problem for the functional equation of all even functions
  \begin{equation}\label{eq:functionalequforeven}
   y(x)= y(-x),\quad  x \in I ,\quad  y(x)=y^0(x), x \in S
  \end{equation}
  has a solution that is uniquely defined on the whole of $ I $.
\end{theorem}

\begin{proof}
  Let $x  \notin S  $. Then  $ -x  \in S $. Then define $y(x)= y^0(-x), x \notin S $. Therefore  the solution is given on $I$ as
  $$ y(x)= \begin{cases}
              & y^0(x),    \mbox{  if }  x \in S,  \\
              &  y^0 (-x),  \mbox{  if } x \in I \setminus S.
           \end{cases} $$
  \end{proof}

\subsection{The functional equation of all odd function }
For odd function $f$ it is known that $f(0)=0 $, we consider the point other than $0$.

\begin{theorem}
Consider a symmetric set $I=(-a,a), a> 0 $, or  $ I = \mathbb{R}$ when $a= \infty $.
Let $\mathcal{C} $ is a family of subsets $S$  of $I$ satisfying the following conditions.

  \begin{itemize}
  \item For each  $S \in \mathcal{C} $, for each   $x \in I, x \neq 0  $  either $x \in S $ or $-x \in S $.
  \end{itemize}
  Then  for any arbitrary function $y^0$  defined on some given set  $S \in \mathcal{C} $, the initial value problem for the functional equation of all odd functions
  \begin{equation}\label{eq:functionalequforodd}
   y(-x)= -y(x),\quad  x \in I ,\quad  y(x)=y^0(x), x \in S
  \end{equation}
  has a solution that is uniquely defined on the whole of $ I $.
\end{theorem}

\begin{proof}
  Let $x  \notin S  $. Then  $ -x  \in S $. Then define $y(x)= -y^0(-x), x \notin S $. Therefore  the solution is given on $I$ as
  $$ y(x)= \begin{cases}
              & y^0(x),    \mbox{  if }  x \in S,  \\
              &  -y^0 (-x),  \mbox{  if } -x \in S ,\\
              & 0, \mbox{if } x=0.
           \end{cases} $$
  \end{proof}

\section{The functional equation $ y(3x)=y(x) + y(2x)$}

In this section we set the  possible set on which the initial data is defined so that the initial value problem the functional equation
\begin{equation}\label{eq:x2x3x}
  y(3x)=y(x)+y(2x),
\end{equation}
exists and is unique.
\begin{definition}
Equation (\ref{eq:x2x3x}) can be rearranged into the form
\begin{equation}\label{eq:combinatorialform}
   y(x)  = y\left(\frac{1}{3}x \right)+y\left(\frac{2}{3}x \right)= \left( S^{\frac{1}{3}}+S^{\frac{2}{3}}\right)y(x).
\end{equation}
 We call the equation (\ref{eq:combinatorialform}) the \emph{ iteration formula for forward extension}..
\end{definition}

\begin{definition}
  Equation (\ref{eq:x2x3x}) can be rearranged into the form
\begin{equation}\label{eq:forwardx2x3x}
     y(x)= y(3x)-y(2x)= (S^3-s^2)y(x).
 \end{equation}
Equation (\ref{eq:forwardx2x3x}) is called \emph{ iteration formula for backward extension}.
\end{definition}

\begin{theorem}
Let $\epsilon$  be any positive number. The functional equation (\ref{eq:x2x3x}), with the initial data  $y(x)=y_0(x) $ defined on the interval $I_0= [\epsilon, 3\epsilon)$, has a unique solution on the interval $[0,\infty)$.
\end{theorem}

\begin{proof}
 Consider the interval $I_1:= [3\epsilon, \frac{9}{2}\epsilon) $. Let $x\in I_1 $. Then $\frac{x}{3}  \in [\epsilon , \frac{3}{2}\epsilon ) \subset I_0  $, and $\frac{2x}{3}  \in [2\epsilon,  3 \epsilon ) \subset I_0  $. Let us denote the  calculated value of $y$ defined on $I_1$ by $y_1$. Consequently
   $$y_1(x)=y\left(\frac{x}{3}\right)+y\left(\frac{2x}{3}\right)= y_0\left(\frac{x}{3}\right)+y_0\left(\frac{2x}{3}\right) $$

    Next consider the interval $I_2:= [ \frac{9}{2}\epsilon, \frac{27}{4}\epsilon  ) $. Let $x\in I_2 $. Then $\frac{x}{3}  \in [ \frac{3}{2}\epsilon, \frac{9}{4}\epsilon,   ) \subset I_0  $, and $\frac{2x}{3}  \in [3\epsilon,  \frac{9}{2} \epsilon ) \subset I_1  $. Let us denote the  calculated value of $y$ defined on $I_2$ by $y_2$. Consequently
   $$y_2(x)=y\left(\frac{x}{3}\right)+ y\left(\frac{2x}{3}\right)= y_0\left(\frac{x}{3}\right)+ y_1\left(\frac{2x}{3}\right) $$
   We may continue in this manner, defining $y_n$, as the computed value of $y$ on the interval
   $$ I_n:= \left[\frac{3^n}{2^{n-1} }\epsilon ,   \frac{3^{n+1}}{2^{n} }\epsilon \right), n \in \mathbb{N}. $$
   Let $x \in I_{n+1} $. Then $\frac{2}{3}x \in I_n $, and  $\frac{1}{3}x \in \bigcup \limits_{k=i}^n I_k $. Thus    The calculated value $y_{n+1}$ is some linear combination of the preceding calculated values:
   $$ y_{n+1}(x)= \sum_{i=0}^{n}\alpha_i y_i(x) $$

    Let $I_{-1}:= [\frac{\epsilon}{2}, \epsilon )$. If $ x \in I_{-1} $ then $ 3x \in  [\frac{ 3 \epsilon}{2}, 3\epsilon )   \subset I_0 $, and $ 2x \in [\epsilon, 2\epsilon )\subset I_0  $. Consequently
   $$ y(x)=y(3x)-y(2x)=y_0(3x)-y_0(x)  . $$
   More generally let us denote
    $$ I_{-k}:= [ 2^{-k} \epsilon, 2^{1-k} \epsilon),\quad k \in \mathbb{N }. $$
    Suppose  that we have already calculated the values of $y$ on each of the intervals $ I_{-k}, \, k=1,2,.., n $. Then for $x \in I_{-(n+1)} $, we have $  2x \in   I_{-n} $, and $ 3x  \in   I_{-n} \cup I_{1-n} $. Accordingly
    $$ y_{-(n+1)}(x) = \begin{cases}
            & y_{-n}(2x)+ y_{-n}(3x),          \mbox{ if } 3x    \in  I_{-n}  \\
            &  y_{-n}(2x)+ y_{1-n}(3x),          \mbox{ if } 3x   \in I_{1-n}.
            \end{cases} $$
    Then we have $$ (0,\epsilon) = \bigcup_{k=1}^\infty I_{-k} $$
    Summarizing these results we get $y$ defined on the interval $[0,\infty)$ as
    $$ y(x) = \sum_{-\infty}^{\infty}y_k(x)\chi_{I_k}(x), \quad y(0)=0 $$
    For the construction of intervals $I_k$, refer to Figure 7.
\end{proof}

\begin{figure}[H]
\centering
\begin{tikzpicture}[scale=2, >=stealth]
    \draw[->] (0,0) -- (3.5,0) node[right] {$x$};
    \draw[->] (0,0) -- (0,5.5) node[above] {$y$};

    \draw[blue, thick] (0,0) -- (3,3) node[above right] {$y=x$};
    \draw[red, thick] (0,0) -- (2.5,5) node[above right] {$y=2x$};
    \draw[green!70!black, thick] (0,0) -- (1.8,5.4) node[above right] {$y=3x$};

    \def\visepsilon{1} 
    \coordinate (P1) at ({\visepsilon/4}, {\visepsilon/2});     
    \coordinate (P2) at ({\visepsilon/2}, {\visepsilon/2});     
    \coordinate (P3) at ({\visepsilon/2}, {\visepsilon});       
    \coordinate (P4) at ({\visepsilon}, {\visepsilon});         
    \coordinate (P5) at ({\visepsilon}, {3*\visepsilon});       
    \coordinate (P6) at ({1.5*\visepsilon}, {3*\visepsilon});   
    \coordinate (P7) at ({1.5*\visepsilon}, {4.5*\visepsilon}); 
    \coordinate (P8) at ({2.25*\visepsilon}, {4.5*\visepsilon}); 

    \coordinate (Ref) at ({\visepsilon}, 0);

    \draw[very thick, magenta] (P1) -- (P2);
    \draw[very thick, magenta] (P2) -- (P3);
    \draw[very thick, magenta] (P3) -- (P4);
    \draw[very thick, magenta] (P4) -- (P5);
    \draw[very thick, magenta] (P5) -- (P6);
    \draw[very thick, magenta] (P6) -- (P7);
    \draw[very thick, magenta] (P7) -- (P8);

    \foreach \point/\position/\coord in {
        P1/above left/{$(\frac{\epsilon}{4}, \frac{\epsilon}{2})$},
        P2/below left/{$(\frac{\epsilon}{2}, \frac{\epsilon}{2})$},
        P3/above left/{$(\frac{\epsilon}{2}, \epsilon)$},
        P4/below right/{$(\epsilon, \epsilon)$},
        P5/above right/{$(\epsilon, 3\epsilon)$},
        P6/below right/{$(\frac{3\epsilon}{2}, 3\epsilon)$},
        P7/above right/{$(\frac{3\epsilon}{2}, \frac{9\epsilon}{2})$},
        P8/above right/{$(\frac{9\epsilon}{4}, \frac{9\epsilon}{2})$},
        Ref/below/{$(\epsilon, 0)$}}
    {
        \node[circle, fill=black, inner sep=1.5pt] at (\point) {};
        \node[\position, font=\scriptsize] at (\point) {\coord};
    }

    \draw[gray!20, very thin] (0,0) grid (3.5,5.5);
    \draw[thick] ({\visepsilon}, -0.05) -- ({\visepsilon}, 0.05) node[below] {$\epsilon$};
    \draw[thick] (-0.05, {\visepsilon}) -- (0.05, {\visepsilon}) node[left] {$\epsilon$};
    \draw[thick] (-0.05, {3*\visepsilon}) -- (0.05, {3*\visepsilon}) node[left] {$3\epsilon$};
    \draw[thick] (-0.05, {4.5*\visepsilon}) -- (0.05, {4.5*\visepsilon}) node[left] {$\frac{9\epsilon}{2}$};
\end{tikzpicture}
\caption{Graphical representation of  subsequent new intervals that are obtained by forward and backward iteration based on the initial interval $I_{0}(\epsilon)=[ \epsilon, 3 \epsilon)$ .}
\label{fig:epsilon_construction}
\end{figure}
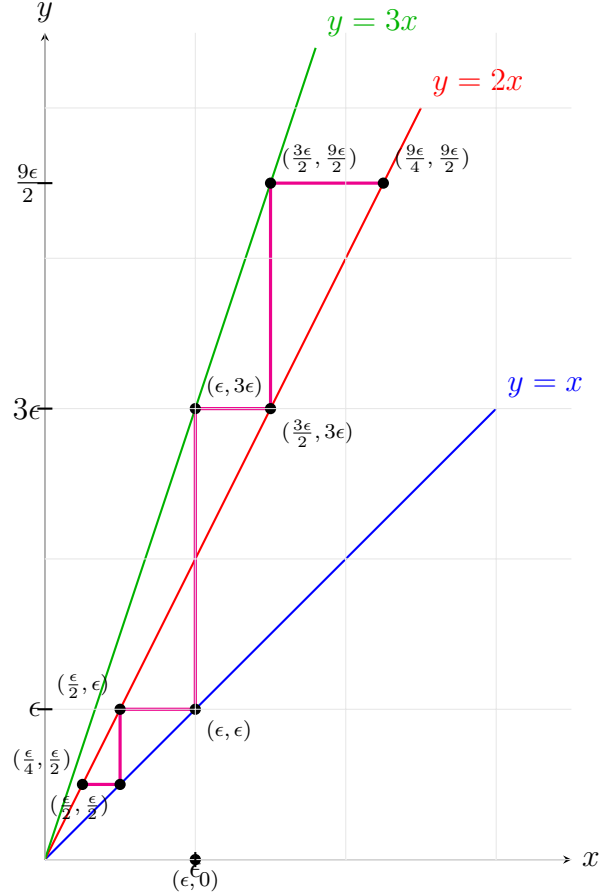
We can also extend the result in a similar manner when the initial set is given within the negative real line, as stated in the next theorem.
\begin{theorem}
Let $\delta$  be any positive number. The functional equation (\ref{eq:x2x3x}), with the initial data  $y(x)=y_0(x) $ defined on the interval $I_0= ( -3\delta,  - \delta ]$, has a unique solution on the interval $ (-\infty, 0] $.
\end{theorem}

\begin{theorem}
  For any $\delta, \epsilon > 0 $, the initial value problem for (\ref{eq:x2x3x})  with initial data $y_0$ defined on the interval of the form $(-3\delta, -\delta] \cup [\epsilon, 3\epsilon)  $ is sufficient to calculate the value of $y$ on the whole of $\mathbb{R}$.
\end{theorem}

%

\begin{theorem}
  For the functional equation (\ref{eq:x2x3x}), if an initial value $y_0$ is set on an interval $[a,b),\,0< a < b $ then we have the following  possible cases:
  \begin{itemize}
  \item If $ b=3a $, then the initial value problem  has unique solution on $[0,\infty) $.
    \item If $ 3a < b $, the initial value problem is \textbf{overdetermined}. That is, while the inital data given on interval $[a,3a)$ is sufficient to calculate $y$ on $[0,\infty)$,  the extra data given on $[3a,b)$ create an inconsistency, unless for the exceptional case where the calculated value of $y$ on $[3a,b)$ coincides with the given data on $[3a,b)$.
    \item If $ b< 3a  $ then the initial value problem is \textbf{underdetermined}. That is  the initial data $y_0$ defined on $[a,b)$ is not enough to guarantee the unique solution of the functional equation on $[0,\infty) $. arbitrary assignment of additional data on $[b,3a)$ yields a different solution on $[0,\infty) $.
        \end{itemize}
\end{theorem}

\section{The principle of exclusion of a neighbourhood of limit points - the PENLP}
\subsection{The Core Concept}
Consider a general functional equation of the form
\begin{equation}\label{eq:mainFE}
F\left(x, y(x), (Ly)(x)\right) = 0,
\end{equation}
where $L$ is an operator (e.g., a shift $y(x+c)$, scaling  $y(bx)$). An initial condition is prescribed on a set $I_0 \subset \mathbb{R} $:
$$ y(x) = y_0(x), \quad \text{for all } x \in I_0.   $$
A fundamental question is: under what conditions on the initial set $I_0$ does this initial value problem have a unique solution? The following concept is crucial for identifying potential obstructions to uniqueness, arising from the equation's inherent structure forcing specific behavior on sequences converging to a point. We assume that the function $y$  can be computed on successive  disjoint intervals (or sets )  $I_k$, based on the initial function $y_0$ defined on $I_0$. Furthermore,  we assume that for all iterations,
 $$ x_{k,s}(x)\in  I_0,\quad  k \in \mathbb{N } \quad  s =1,2,...,n_k, , \quad x \in I_k . $$
  This condition ensures that the points involved in the iteration do not lie outside the initial domain $I_0$ as the domain of $y$ expands after each iteration depending on the initial data. We consider  that the first iteration $y_1$ depends solely on the initial  function $y_0$. However,  the subsequent iterates $y_k$ depends on the preceding iterate functions $y_1,y_2,...,y_{k-1}$. Ultimately, all the iterates  depend on the initial function $y_0$. Formally, we can write:
\begin{align}\label{eq:consequetiveiterations}
    y_1(x) &= f_1\left( y_0(x_{1,1}(x)), y_0(x_{1,2}(x))(x), \dots, y_0(x_{1,n_1}(x)) \right),\quad x\in I_1, \nonumber \\
    y_2(x) &= f_2\left( y_0(x_{2,1}(x)), y_0(x_{2,2}(x)), \dots, y_0(x_{2,n_2}(x)) \right) ,\quad x\in I_2, \nonumber \\
    &\vdots \nonumber \\
    y_k(x) &= f_k\left( y_0(x_{k,1}(x)), y_0(x_{k,2}(x)), \dots, y_0(x_{k,n_k}(x)) \right),\quad x\in I_k.
\end{align}
As, such we extend the interval of existence to newer sets. Let $  I_{\text{max}} $ denote the largest set on which $y$ can be defined by extension, based on the initial function $y_0$  defined on $I_0$. Then,
$$  I_{\text{max}}= \bigcup_{k=0}^{\infty}I_k .$$

\begin{example}
 Consider the IVP for the functional equation (\ref{eq:x2x3x}). With initial function $y(x)=y_0(x)$  on $ [\epsilon, 3\epsilon)$. For $k=1$, we have
  $$ n_1= 2,\, I_1=[3 \epsilon, \frac{9}{2}\epsilon),\, x_{1,1}(x)= \frac{x}{3},\, x_{1,2}(x)= \frac{2x}{3},\, f_2(u,v)=u+v.  $$
  Accordingly,
  \begin{align*}
    f_1(y_0(x_{1,1}(x)), y_0(x_{1,2}(x) ))  & = y_0(x_{1,1}(x))+ y_0(x_{1,2}(x) ),  \\
     & = y_0\left( \frac{x}{3}\right)+   y_0\left( \frac{2x}{3}\right), \quad x \in I_1.
     \end{align*}
\end{example}

\begin{definition}[Limit Point of Constraint] \label{def:constraint}
A point $l \in \mathbb{R} $ is called a \textbf{limit point of constraint} for the functional equation \eqref{eq:mainFE} if there exists a sequence of points in $\mathbb{R }$
$$   \{x_{k,s}(x) \}_{k \in \mathbb{N}},\quad s \in \{1, 2, \ldots n_k \},  $$
 with infinitely many distinct elements such that:
\begin{enumerate}
  \item  $x_{k,s }(x) \to l $ as $n \to \infty$ for all $ x \in I_{\text{max}} $,
  \item  The equation \eqref{eq:mainFE} \textbf{constrains} the behavior of any prospective solution on the sequence $\{x_n\}$. Formally, for any function $y$ satisfying \eqref{eq:mainFE} in a neighborhood of $l$, the values $ \{y(x_n)\}$ are not independent but must satisfy a relation derived from the equation itself.
\end{enumerate}
\end{definition}

%

\begin{theorem} [\textbf{PENLP}] \label{thm:PENLP}

 Let $ l \in \mathbb{R} $ is  a limit point of constraint for the functional equation \eqref{eq:mainFE}. Suppose that the initial value problem for the functional equation  \eqref{eq:mainFE} with an initial function $y_0$ defined on some initial set $I_0$ is given. Then the necessary condition for  existence of the well-defined solution for the initial value problem \eqref{eq:mainFE} is that there exists $\epsilon > 0 $ such that $I_0 \cap B(l,\epsilon) = \emptyset $. That is, the initial set $I_0$ should excludes a small beggerhood of the limit point $l$.
\end{theorem}

\begin{proof}
 We assume that the initial set $I_0$ contains some neighborhood of a limit point $l$. Let $x \in I_{\text{max}}$.  Suppose that there exists $  N \in \mathbb{N} $ such that
$$  x_{k,s}(x) \in B\left( l,\frac{1}{k}\right) \cap I_0, \quad  \forall  k \geq  N,\quad  \forall s \in \{1,2,...,n_k\}. $$
Then
$$ x_{k,s}(x) \to l \quad  \text{ as } \quad k \to \infty. $$
By (\ref{eq:consequetiveiterations}),
$$ y_k(x) = f_k\left( y_0(x_{k,1}(x)), y_0(x_{k,2}(x)), \dots, y_0(x_{k,n_k}(x)) \right),\quad k\geq N. $$
Therefore $y(x)$ can be calculated in an infinitely many different ways depending on  $ k \geq N $. This forces the initial function $y_0$ to satisfy infinitely many constraints, which is almost always impossible. This proves the theorem.
 \end{proof}


\subsection{  The PENLP and the functional equation $y(x+1)= y(bx)$  }

\begin{theorem}
  Let $b \neq -1, 1$. The number $bx^*  = b(b-1)^{-1}$ is a limit point of the functional equation (\ref{eq:xplus1bx})
\end{theorem}

\begin{proof}
\begin{itemize}
  \item  Let  $|b|> 1 $.  Consider the sequence  where $\{x_n(x)\}$, where
   $$ x_n(x)= b^{-n}(x-bx^*)+ bx^*, x \neq bx^*, x_0(x)=x,  n \in \mathbb{ N} . $$
 This sequence of infinitely many distinct elements and converging to $bx^* $. In addition, $y(x_n(x)= y(x)$ for all $n\in \mathbb{N} .$

  \item Let $ 0 < |b| < 1 $.  Consider the sequence $\{x_n(x)\}$, where
   $$x_n(x)= b^{n}(x-bx^*)+ bx^*, x \neq bx^*, x_0(x)=x,  n \in \mathbb{ N}.$$
   This sequence of infinitely many distinct elements and converging to $bx^* $. In addition, $y(x_n(x)= y(x)$ for all $n\in \mathbb{N} .$
\end{itemize}
 \end{proof}

\subsection{  The PENLP and the functional equation $y(x)= y(2x)$  }

\begin{theorem}
  The point $x=0 $ is the limit point of the constraint set the functional equation $y(x)= y(2x)$.
\end{theorem}

\begin{proof}
Consider the sequence $\{x_n(x)\}$, where $x_n(x)=x2^{-n}, x \neq 0 $. This sequence consists of infinitely  many distinct elements and converge to $0$ as $n$ approaches to infinity. Additionally, we have $y(x_n(x))=y(x)$ for all $n\in \mathbb{N}$. Hence, $0$ is the a limit point, and of course, the only limit for the functional equation.
\end{proof}
 In the next example, we examine the situation that arises when considering an initial value problem defined on some initial set.
\begin{example}
  Let $\epsilon >0 $ be arbitrary. Consider the initial value problem for the functional equation
  \begin{equation}\label{eq:IVPynot}
     y(x)= y(2x),   \quad y(x)= y_0(x), -\epsilon < x < \epsilon .
     \end{equation}
Let $|x| \geq \epsilon$. There exists $ N(\epsilon) \in \mathbb{N} $ such that $x2^{-n}\in (-\epsilon, \epsilon), n \geq N(\epsilon) $.  Then
     $$ y(x)= y(x2^{-n})=y_0(x2^{-n} ), n \geq N(\epsilon). $$

     However, this does not generally hold for an arbitrary function $y_0 $ defined on $ (-\epsilon.  \epsilon )$; in most cases, it is inconsistent. It may, however, hold true when  $y_0$ is a constant.
\end{example}

\subsection{  The PENLP and the functional equation $y(3x)=y(x)+y(2x)$  }

We have seen that the initial value problem for  the functional equation (\ref{eq:x2x3x})  with initial data $y_0$ given on some initial set of the form $I_0:= [\epsilon,3\epsilon)$ is sufficient to uniquely calculate $y$ on the entire interval $  I_{\text{max}}=[0,\infty)$. Such an interval $I_0$ is of length $2\epsilon $ and can be made of arbitrarily small length by considering  small positive number $\epsilon$. Next, we see the problem that arises if we consider the initial value problem for  the functional equation (\ref{eq:x2x3x}) with an initial data $y_0$ defined on some interval of the form $(0,\epsilon)$ rather than $(\epsilon, 3\epsilon)$. For any $n \in \mathbb{N} $, the backward iteration equation (\ref{eq:combinatorialform}) can be written as
\begin{equation}\label{eq:npowerequation}
 y(x)  =\left( S^{\frac{1}{3}}+S^{\frac{2}{3}}\right)^n y(x) = \sum_{r=0}^{n}\binom{n}{r}y \left(\frac{2^r}{3^n}x\right) .
\end{equation}
For any given $ \epsilon > 0 $ (however small), and given $x>0$ there exit $n \in \mathbb{N} $ such that
  \begin{equation}\label{eq:nofepsilon}
    0 < \left(\frac{2}{3}\right)^n x < \epsilon
  \end{equation}
Consequently
\begin{equation}\label{eq:allnodesinclude}
  0 < \frac{2^r}{3^n} x < \epsilon, \quad r=0,1,2,...,n.
\end{equation}


Hence by (\ref{eq:allnodesinclude}) and the given initial data $y_0$, the solution $y$ can be defined on the interval  $[0, \infty ) $ as
\begin{equation}\label{eq:solnIVPx2x3x}
 y(x)= \begin{cases}
            & y_0(x), \mbox{  if } 0 < x < \epsilon,  \\
            & \sum_{r=0}^{n}\binom{n}{r}y_0\left( \frac{2^r}{3^n}x\right),  \mbox{ if } x \geq \epsilon .
         \end{cases}
\end{equation}
Let $n^* $ be the smallest natural number for which (\ref{eq:nofepsilon}) holds true. Then (\ref{eq:nofepsilon}) holds true for all natural numbers $n > n^* $ and hence the solution (\ref{eq:solnIVPx2x3x}) may be calculated using  any $n > n^*$. However (\ref{eq:solnIVPx2x3x}) may not be well-defined for arbitrary initial function $f$ defined on the interval $(0,\epsilon )$, as we may get two different values  of $y(x)$ if we use two  possible values $n \in \mathbb{N} $. We illustrate this condition by the next example.
\begin{example}
  Consider the initial value problem for the functional equation
  \begin{equation}\label{eq:IVPyequalsxsquared}
     y(x)= y \left( \frac{2}{3}x\right) + y \left( \frac{1}{3}x\right), x \geq 0,   \quad y(x)=x^2, 0< x < 1 .        \end{equation}

 Suppose that we want to evaluate $y(1.2)$. So, $x=1.2, \frac{2}{3} x= 0.8, \frac{1}{3} x= 0.4 $. Now  for $y(1.2)$ calculated according to (\ref{eq:solnIVPx2x3x}) using $n=1 $ yields
 $$ y(1.2)=y(0.8)+y(0.4)= 0.8^2+0.4^2 = 0.80 .$$
 On the other hand, if we calculate $y(1.2)$ according to (\ref{eq:solnIVPx2x3x}) using $n=2 $ we get
 $$ y(1.2) =\left (\frac{1.2}{9}\right)^2 + 2 \left(\frac{2.4}{9}\right)^2  + \left(\frac{4.8}{9}\right)^2= \frac{101}{225}. $$
\end{example}
Hence the value of $y(1.2)$ calculated according to (\ref{eq:npowerequation}) using $n=1$ and $n=2$  are different.
 Rather than $y(x)=x^2, 0< x < 1 $, if the initial   function were defined as  $f(x)= mx $ for some fixed $m \in \mathbb{R} $, we get
$$ y(x) =  \sum_{r=0}^{n}\binom{n}{r}f\left( \frac{2^r}{3^n}x\right) =  \sum_{r=0}^{n}\binom{n}{r}  \frac{2^r}{3^n} mx = mx,  $$
showing that the result is independent of $n$.

\section*{ Discussion   of the results}

In this paper, we studied several functional equations with initial functions defined on some prescribed set. We generalized our discussion to the broader class of functional equations and formulated the necessary conditions that should be considered when selecting the initial set $I_0$ on which the initial function is defined, so that we can uniquely extend the initial function to the largest possible set $I_{\text{max}}$.

The PENELP is the acronym for the \emph{Principle of the Exclusion of the Neighborhood} of the limit point of the functional equation. The functional equations under discussion may not have limit points in the sense that, for any point $l$, we can consider an initial set $I_0$ that contains $l$, and then extend the initial function to a larger set uniquely.

The study of functional equations has several applications, and their investigation is worthwhile in the same manner as the study of differential and integral equations. In a broader sense, functional equations can be considered as generalizations of difference equations, differential equations, and integral equations.

Finally, we would like to include one possible application or physical modeling related to one of the functional equations discussed in this paper. Consider the functional equation (\ref{eq:x2x3x}).

Suppose there are three coaxial cylindrical wires with interior radii $r_1 = \sqrt{\frac{1}{\pi}}$, $r_2 = \sqrt{\frac{2}{\pi}}$, and $r_3 = \sqrt{\frac{3}{\pi}}$, so that the cross-sectional areas of the concentric circles are 1, 2, and 3 square units, respectively, from the inside out.

Let the length $x$ run from one end of the wire, which we set at $x=0$, to the other end. The volumes of the wires then are $x$, $2x$, and $3x$, respectively. If $y$ is a function defined on $[0, \infty)$, we may be interested in calculating the function that depends on the cumulative volume of the wire according to the functional equation (\ref{eq:x2x3x}).

The initial value problem can be interpreted accordingly when the specific value of the function $y$ is given between two cross sections, say at $x=0$ and $x=3$.

\section*{Declaration statements}

\subsection*{Conflict of Interests}

The author declares that there are no conflicts of interest regarding the publication of this paper.

\subsection*{Acknowledgment}
I would like to thank Professor Arthur Rubin for sharing his valuable ideas via LinkedIn messaging.

\subsection*{Funding}
This research work has not been funded by any institution or individual.

\subsection*{ Author Contribution}
The corresponding author is the sole creator and author of the paper.

\subsection*{Data availability}
This paper uses no external data other than the literature references mentioned here.

\subsection*{Ethical Approval}

This paper contains novel results, adheres to the ethical guidelines of the publisher, and has not been submitted elsewhere.

\end{document}